\documentclass[12pt,reqno]{amsart}

\usepackage{amsthm,amsmath,amssymb}
\usepackage{graphicx,color}
 \usepackage[colorlinks=true]{hyperref}

\hypersetup{urlcolor=blue, citecolor=red}

\usepackage{enumitem}
\usepackage{yhmath}

\newtheorem{theorem}{Theorem}[section]

\newtheorem{corollary}[theorem]{Corollary}

\newtheorem{definition}[theorem]{Definition}

\newtheorem{lemma}[theorem]{Lemma}

\newtheorem{proposition}[theorem]{Proposition}
\newtheorem{remark}[theorem]{Remark}

\title{The elastica problem under area constraint}
\author[V. Ferone, B. Kawohl, C. Nitsch]
       {Vincenzo Ferone, 
       Bernd Kawohl, 
       Carlo Nitsch
       }
\address[V.~Ferone]{Universit\`a degli Studi di Napoli Federico II (Italy).}
\email{ferone@unina.it}
\address[B.~Kawohl]{Universit\"at zu K\"oln (Germany).}
\email{kawohl@math.uni-koeln.de}
\address[C.~Nitsch]{Universit\`a degli Studi di Napoli Federico II (Italy)\\ 
 Universit\"at zu K\"oln (Germany).}
\email{c.nitsch@unina.it}

\date{\today}

\begin{document}
\maketitle

\begin{abstract}
We show that the elastic energy $E(\gamma)$ of a closed curve $\gamma$ has a minimizer among all plane simple regular closed curves of given enclosed area $A(\gamma)$, and that the minimum is attained for a circle. The proof is of a geometric nature and deforms parts of $\gamma$ in a finite number of steps to construct some related convex sets with smaller energy.
\end{abstract}

\section{introduction}

The origin of the elastica problem can be traced back to the birth of the calculus of variations \cite{Levien2008,Truesdell1983}. As early as 1691 James Bernoulli provided the first example of a mathematical formulation for the special case of a bending lamina fixed at one end point and subject to a load at the other end point. Half a century later Daniel Bernoulli, in the attempt to minimize the bending energy of inextensible wire, proposed a refined version of the problem.
He wrote to Euler because of his expertise, that nobody was indeed more qualified than him in using the ``methodum isoperimetricorum" (isoperimetric method). It was the year 1744 and the calculus of variations was in its early stages when Euler wrote his treatise on variational techniques (\emph{Methodus inveniendi lineas curvas maximi minimive proprietate gaudentes, sive solutio problematis isoperimetrici lattissimo sensu accepti}). There he systematically investigated and unified many celebrated variational problems that are today considered to be mathematical folklore. In his masterpiece he devoted an entire chapter, \emph{De Curvis Elasticis}, to working out a complete characterization of those curves which are solutions to the elastica problem. 

Since then the curves are called ``elasticae''. 
The subject of elasticae has attracted generations of mathematicians as well as physicists and it is still an active field of investigation.

In a very broad sense elasticae are curves which are stationary points of the elastic energy functional. The elastic energy of a smooth curve $\gamma$ is the integral of its squared curvature. The original formulation of the Elastica Problem was worked out for planar curves but later on it was generalized to several dimensions as well as to curves on manifolds \cite{Koiso1992,Langer1984,Langer1984_1,Singer2008}. An elastica can be open or closed. Generally speaking the word elastica is used to denote a stationary point of the elastic energy for given length. In contrast to this elasticae without length constraint are  sometimes called free elasticae. One of the first studies on these was done by Radon \cite{Radon1928}, 

When dealing with closed elasticae several different constraints can be considered. In the planar case, for instance, together with fixed length an additional area constraint has been considered among others by \cite{Arreaga2002,Bianchini2014,Veerapaneni2009,Watanabe2000,Watanabe2002}. Another possibility is to restrict the set of curves to those which lie inside a given domain, see for instance \cite{Dondl2011}. 
In \cite{Arreaga2002,Giomi2012} the authors considered the problem of minimizing a functional involving the elastic energy of an inextensible closed wire and the area of the minimal surface spanned by such a curve, thus combining two of the most prominent problems in the calculus of variations into what we can refer to as the Plateau-Elastic problem. 

In curve shortening flow, where the speed of the shrinking curve is proportional to its curvature, the rate of decrease in length is the elastic energy of the curve. As shown in \cite{Gage1983}, an optimal bound on the elastic energy is crucial to prove that the isoperimetric deficit decreases as the curve shrinks to a point. 

The study of the gradient flow of the elastic energy functional gives rise to 
curve straightening flow. This subject is under intensive investigation, see for example \cite{Dziuk2002,Langer1985,Linner1998,Lin2012,Okabe2007,Wen1995}. 

The elastic energy also appeared in the study of the spectrum of the Laplacian under Dirichlet as well as Neumann boundary condition. In fact, for a planar smooth bounded set the elastic energy of its boundary is one of the leading coefficients in the power series expansion of the associated heat kernel \cite{Smith1981,Watanabe2000,Watanabe2002}. For this reason the minimization of elastic energy plays a role in the characterization of those domains that are spectrally determined. Moreover, in differential geometry there is a close connection between the Willmore functional for a surface of revolution in $\mathbb{R}^3$ and the elastic energy of curves immersed in the hyperbolic plane, see for example \cite{Langer1984_2, Willmore2000}.

In spite of all the literature on the subject, it has been noted in \cite{Bianchini2014} that there exists a very simple question pertaining to elasticae that is apparently still unsolved.
\\

\noindent{\bf Open problem.} \emph{Is there a curve that minimizes the elastic energy in the class of all simple smooth planar curves that enclose a given area and, if so, is this curve a circle?}
\\


For a precise mathematical statement let $(\bf{e_1},{\bf e_2})$ be the canonical orthonormal basis of $\mathbb{R}^2$. For a smooth regular planar curve $\gamma:[0,L]\to \mathbb{R}^2$ parametrized by arc length $s$, by convention we define the normal vector ${\bf n}(s)$ so that $\bf{n}(s)$ and  the unit tangent vector ${\bf t}(s)=\gamma'(s)$ form for all $s\in [0,L]$ a basis $({\bf t},{\bf n})$ that has the same orientation as the basis $(\bf{e_1},{\bf e_2})$. If $\gamma\in W^{2,2}([0,L]; \mathbb{R}^2)$, we can define the scalar signed curvature $k(s)$ as
$$\frac{d}{ds}{\bf t}(s)=:k(s)\,{\bf n}(s),$$

\noindent and the elastic energy $E(\gamma)$ by
$$E(\gamma):=\frac{1}{2}\int_\gamma k(s)^2\ ds.$$

If the curve is simple and $\gamma(0)=\gamma(L)$, we denote by $\Omega_\gamma$ the bounded open set of $\mathbb{R}^2$ with boundary $\gamma=\partial\Omega_\gamma$ and by $Area(\Omega_\gamma)$ or more simply by $A(\gamma)$ the area of $\Omega_\gamma$. Throughout the paper we also choose the arc length representation in which the normal to $\gamma$ points into $\Omega_\gamma$. Such a parametrization is commonly called {\emph{counterclockwise orientation} (or \emph{positive orientation}).
In the following we speak of a curve $\gamma:[0,L]\to \mathbb{R}^2$ as a \emph {regular  closed curve} if $\gamma(0)=\gamma(L)$ and $\gamma'(0)=\gamma'(L)$.

Our main result is the following. 

\begin{theorem}\label{main_th}
For any given simple regular closed curve $\gamma$ in $\mathbb{R}^2$, with $\gamma\in W^{2,2}$, we have
\begin{equation}\label{isopineq}
A(\gamma)E(\gamma)^2\ge \pi^3,
\end{equation}
and equality holds if and only if $\gamma$ is a circle.
\end{theorem}

We should mention that  Theorem \ref{main_th} is no longer true if $\Omega_\gamma$ is not simply connected. Any thin annulus can serve as a counterexample.

However, a special case of Theorem \ref{main_th} has recently been pointed out in \cite{Bianchini2014} as being a consequence of a famous result of Gage.
\begin{proposition}\label{pr_gage}
For any given simple closed curve $\gamma$ in $\mathbb{R}^2$, boundary of a {\bf convex} set, with $\gamma$ parametrized by arc length of class $C^{1}$ piecewise $C^2$, we have$$A(\gamma)E(\gamma)^2\ge \pi^3,$$
and equality holds if and only if $\gamma$ is a circle.
\end{proposition}
Note that no convexity is required in Theorem \ref{main_th}. 
The passage from convex $\Omega_\gamma$ to simply connected $\Omega_\gamma$ is not trivial. We will show a counterexample in Figure \ref{fig_henrot} in which, for example, passing from $\Omega_\gamma$ to its convex hull brings an increase of the product $A(\cdot)E(\cdot)$.

As pointed out in \cite{Bianchini2014} Proposition \ref{pr_gage} is essentially due to Gage \cite{Gage1990,Gage1983}. In fact, while Gage derived 
\begin{equation}\label{gage}
 2 E(\gamma)\geq \pi\frac{L}{A(\gamma)}
\end{equation} 
for convex sets $\Omega_\gamma$, an inequality which fails to carry over to nonconvex ones, (\ref{isopineq}) follows from (\ref{gage}) via the classical isoperimetric inequality, because
$$A(\gamma)E(\gamma)^2\geq\frac{\pi^2}{4}\frac{L^2}{A(\gamma)}\geq\pi^3.$$
\medskip

Before entering into a description of our approach to the problem, we would like to explain some of the difficulties that we had to overcome and discuss a few important consequences of Theorem \ref{main_th}.

First of all let us consider the classical problem of minimizing elastic energy in the class of closed curves of fixed length $L$.
The circle is the only minimizer. This fact is fairly easy to prove. Assuming a regular closed curve $\gamma\in W^{2,2}([0,L]; \mathbb{R}^2)$ (not necessarily simple) is parametrized by arc length, the following H\"older inequality
\[
\left(\int_\gamma k(s)\,ds\right)^2 \le L \int_\gamma k(s)^2\,ds
\]
and the Gauss--Bonnet Theorem yields
\begin{equation}\label{eq_fixedL}
L\,E(\gamma)\ge 2\pi^2.
\end{equation}
Moreover, equality holds if and only if $k$ is constant.

Actually a more refined investigation carried out for instance in \cite{Avvakumov2012,Sachkov2012} has shown that there exist only two stationary points: the circle, and the so-called Bernoulli $\infty$-shape (or figure 8 shape) elastica. Moreover in \cite{Enomoto1997} the author studies a quantitative version of \eqref{eq_fixedL} in the spirit of Bonnesen-type isoperimetric inequalities as described in \cite{Osserman1979}. Among other things he proves that for a planar simple curve $\gamma$ of length $L$ we have
\begin{equation}\label{eq_enomoto}
L E(\gamma) - 2\pi^2\ge \frac{\pi^2(R-r)^2}{L^2},
\end{equation}
where $r$ and $R$ are the inradius and outer radius of $\Omega_\gamma$.

We emphasize that in view of the classical isoperimetric inequality $L\geq\sqrt{4\pi A}$ our new inequality \eqref{isopineq} is stronger than \eqref{eq_fixedL}. 

Indeed, in analogy to the   {\em isoperimetric deficit} 
$$\delta L(\gamma)=\frac{L - (4\pi A)^{1/2}}{(4\pi A)^{1/2}}$$
we can define the {\em elastic deficit} as
$$\delta E(\gamma):=\frac{L E(\gamma) - 2\pi^2}{2\pi^2}$$
and arrive at
\begin{corollary}
For any given simple regular closed curve $\gamma$ in $\mathbb{R}^2$, with $\gamma\in W^{2,2}$, it holds
\begin{equation*}
\delta E(\gamma)\ge \delta L(\gamma).
\end{equation*}
\end{corollary}
Moreover, by combining \eqref{isopineq} with the Bonnesen-type inequality  in \cite[Theorem 4]{Osserman1979} 
$$L^2-4\pi A\ge\pi^2(R-r)^2$$
we can also improve \eqref{eq_enomoto} in the following way.
\begin{corollary}
For any given simple regular closed curve $\gamma$ in $\mathbb{R}^2$, with $\gamma\in W^{2,2}$, we have
\begin{equation*}
L E(\gamma) - 2\pi^2\ge \frac{\pi^4(R-r)^2}{L^2}.
\end{equation*}
\end{corollary}

\begin{figure}
\def\svgwidth{4.in}
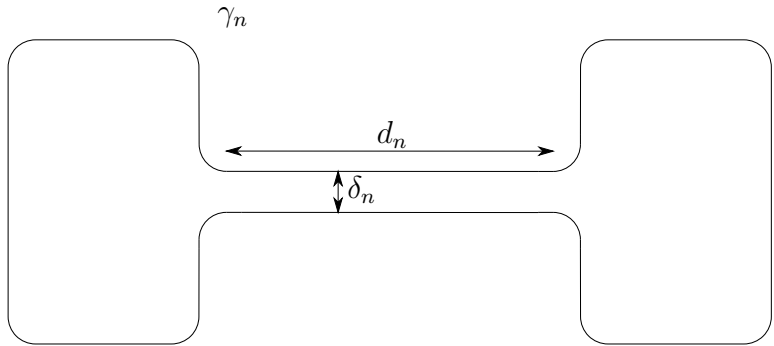\caption{The curve $\gamma_n$ is made of twelve line segments and twelve copies of a quarter of a circle. The length of the line segments may vary with $n$, while the radii of the circular parts remain fixed. If $d_n$ diverges as $n\to\infty$, it is possible to choose $\delta_n\to 0$ as $n\to\infty$ so that both elastic energy and area remain constant along the sequence of curves $\gamma_n$.}\label{fig_henrot}
\end{figure}
Fixing the length of admissible curves obviously implies their equiboundedness (up to translations), and this simplifies any compactness arguments. Without a length constraint, the situation becomes more complicated. To see this 
let us consider the shape derivative of the functional $E(\gamma)$ and formally obtain the Euler-Lagrange equation for a stationary point  under perturbations which are only area preserving (see for instance \cite{Bianchini2014}). We have that
\begin{equation}\label{eq_formal}
k''(s)+\frac{1}{2}k(s)^3=const.
\end{equation}
While it may be possible to characterize closed curves which satisfy \eqref{eq_formal}, one has to keep in mind that we cannot conclude that minima of $E$ correspond to stationary points unless we first prove compactness of an energy minimizing  sequence in the class of simple curves. Observe that the equi-boundedness of both area and elastic energy is in general not sufficient to have an equibounded (in diameter) sequence of curves as shown for instance in the example in Figure \ref{fig_henrot}. We have learned this example from a lecture of A.Henrot. Furthermore, even if we were able to exhibit a minimizing sequence with bounded diameter, we would have to exclude that pinching occurs in order to have a convergence to a simple curve. In principle, as equation \eqref{eq_formal} does not take into account the simplicity of the curve $\gamma$, one can try to gain compactness by relaxing the admissible curves to the class of all closed (not necessarily simple) curves. However, this creates the problem of properly defining the area enclosed by self intersecting curves. If the Gauss-Green formula for the area is used, negative values of the area are allowed and \eqref{isopineq} becomes meaningless. Other choices which are not consistent with the Gauss-Green formula may result in an area functional that is not differentiable with respect to smooth deformations of $\gamma$. Then condition \eqref{eq_formal} also cannot follow in a rigorous way. 

We also emphasize that in order to gain compactness, passing from $\gamma$ to its convex hull $co\,\gamma$ does not appear to be a feasible strategy since the very same example in Figure \ref{fig_henrot} provides a sequence $\gamma_n$ along which 
$A(\gamma_n)E(\gamma_n)^2$ is equi-bounded while $A(co\,\gamma_n)E(co\,\gamma_n)^2$ diverges as $n\to\infty$.

Let us mention yet another consequence of Theorem \ref{main_th} that concerns the curve shortening flow. It is well known \cite{Gage1986} that a planar simple curve shrinks to a point under curvature flow and remains simple while asymptotically converging to a circle. If the initial area is $A_0$, the extinction time is $\displaystyle T=\tfrac{A_0}{2\pi}$ and the area is $A(t)=A_0-2\pi t$ for $t\in[0,T]$. The length of the curve shrinks in such a way that
$\tfrac{d}{dt}L(t)=-2E(t)$. Therefore, we have the following:
\begin{corollary}
If a simple closed curve evolves with time according to the curvature flow, its length $L(t)$ satisfies the following inequality
$$\frac{d}{dt}L\le-2\sqrt\frac{\pi^{3}}{A_0-2\pi t},$$
for all positive $t< \frac{A_0}{2\pi}$ where $A_0$ is the area initially enclosed by the curve. Moreover, equality holds for circles.
\end{corollary}

\section{Preliminary results and strategy of the proof}

As a first step we observe that it suffices to prove Theorem \ref{main_th} for more regular sets.
Indeed by standard results on smooth approximation of Sobolev functions (see also \cite{Enomoto2013}) we have
\begin{proposition}\label{approx}
Let $\gamma\in W^{2,2}$ be a simple regular closed curve in $\mathbb{R}^2$ parametrized by arc length, then for all $\varepsilon>0$ there exists a closed simple curve $\gamma_\varepsilon$ which is $C^1\cap (\mbox{piecewise } C^2)$, such that the curvature changes sign a finite number of times, and such that 
$$A(\gamma)E(\gamma)^2\ge A(\gamma_\varepsilon)E(\gamma_\varepsilon)^2 -\varepsilon.$$
\end{proposition}


Before giving further details let us consider two points $p$ and $q$ on $\gamma$ so that for the given
orientation $p$ precedes $q$. The subset of $\gamma$ connecting these two points we denote as the arc $\wideparen{pq}$.
When $\gamma(0)=\gamma(L)$ and the curve is positively oriented if we travel along $\gamma$ in the direction of the tangent we cover the path counterclockwise periodically and we can define two arcs connecting $p$ and $q$. The arc $\wideparen{pq}$ where $p$ precedes $q$ and the arc $\wideparen{qp}$ where $q$ precedes $p$. 
The symbol $\overline{pq}$ on the other hand will represent chord or the straight line segment with endpoints $p$ and $q$.  Notice that $\overline{pq}=\overline{qp}$ while $\wideparen{pq}\ne\wideparen{qp}$. 

Now we can introduce the notion of an arc  \emph{holding a convex set}.

\begin{definition}
On a given simple closed curve $\gamma$ in $\mathbb{R}^2$, with $\gamma$ parametrized by arc length of class $C^{1}$ piecewise $C^2$, we say that the arc $\wideparen{pq}$ {\emph{holds an open convex set $K$}} (see Figure \ref{fig_holds}) if:

\noindent
$\bullet$ 
the convex set $K\subseteq \Omega_\gamma$ is bounded by the segment $\overline{pq}$ and the arc $\wideparen{pq}$, i.e. $\partial K=\overline{pq}\cup \wideparen{pq}$
\noindent
$\bullet$ the tangents to $\gamma$ at $p$ and $q$ are parallel

\end{definition}

\begin{figure}
\def\svgwidth{3.in}
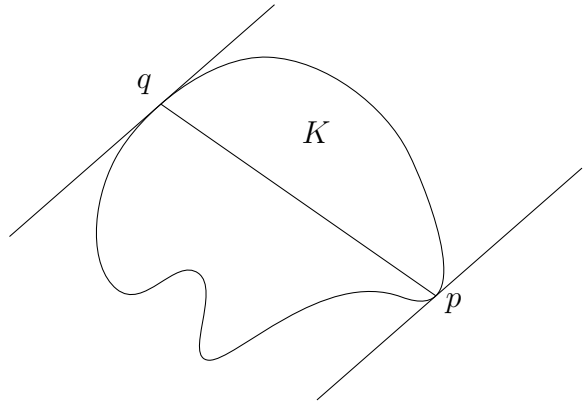\caption{A curve $\gamma$ such that the arc $\wideparen{pq}$ holds the convex set $K$.}\label{fig_holds}
\end{figure}


\noindent 
The following fundamental proposition provides a significant improvement of Proposition \ref{pr_gage}.

\begin{proposition}\label{two_convex}
Let $\gamma$ be a simple regular closed curve in $\mathbb{R}^2$, counterclockwise parametrized by arc length, and of class $C^{1}$ piecewise $C^2$. If $\gamma$ contains two arcs holding two disjoint convex sets, then
$$A(\gamma)E(\gamma)^2\ge \pi^3.$$
\end{proposition}
\begin{proof}
Let $\wideparen{pq}$ and $\wideparen{p'q'}$ be such arcs (see for example Figure \ref{fig_osso}). Let us consider the set $\Omega$ whose boundary is made of the arc $\wideparen{pq}$ and a copy of the same arc rotated by the angle $\pi$ around the center of the segment $\overline{pq}$. The fact that the tangents on $p$ and $q$ are parallel guarantees that the set  $\Omega$ is convex. Then by Proposition \ref{pr_gage} we have that
$$Area(\Omega) \left(2E(\wideparen{pq})\right)^2 \ge \pi^3.$$
Analogously for the arc $\wideparen{p'q'}$ we can construct a convex set $\Omega'$ so that $$Area(\Omega') \left(2E(\wideparen{p'q'})\right)^2 \ge \pi^3.$$
We observe now that $Area(\Omega)+Area(\Omega')\le 2A(\gamma)$ and $E(\wideparen{pq})+E(\wideparen{p'q'})\le E(\gamma)$.
Setting $E(\wideparen{pq})=:t^{-1}$ and $E(\wideparen{p'q'})=:t'^{-1}$ and applying the elementary inequality between arithmetic and geometric means leads to
$$A(\gamma)E(\gamma)^2\ge \frac{\pi^3}{8}\left(t^2+t'^2\right)\left(\frac{1}{t}+\frac{1}{t'}\right)^2\ge {\pi^3}\left(\frac{t^2+t'^2}{2t t'}\right) \ge \pi^3.$$
\end{proof}

\begin{figure}
\def\svgwidth{3.in}
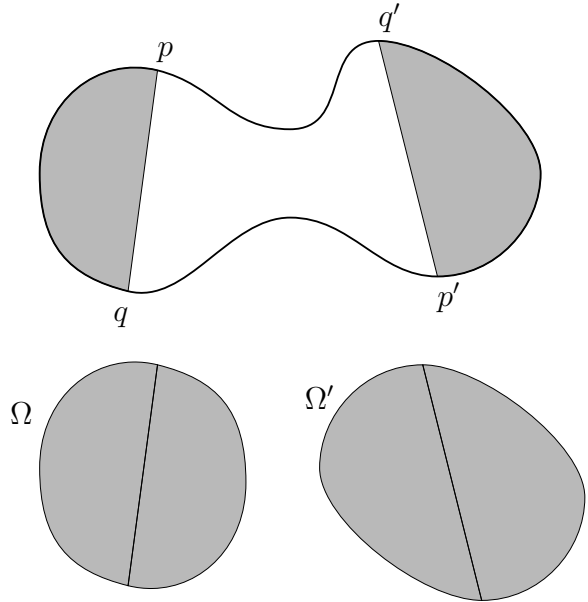\caption{A curve $\gamma$ satisfying hypotheses of Proposition \ref{two_convex}, and the two convex sets $\Omega$ and $\Omega'$ used in the proof.}\label{fig_osso}
\end{figure}

The notion of an arc holding a convex set can be extended to more general curves which are not everywhere smooth. The most important fact is that at least the arcs involved have to be of class $C^1$ and piecewise $C^2$. Therefore from the proof of the previous proposition it is possible to deduce:

\begin{corollary}\label{c_two_convex}
If two sufficiently smooth arcs, say $\wideparen{pq}$ and $\wideparen{p'q'}$, not necessarily belonging to the same closed curve, hold two disjoint convex sets $\Omega$ and $\Omega'$, then
$$\left(A(\Omega\cup\Omega')\right)\left(E(\wideparen{pq})+E(\wideparen{p'q'})\right)^2\ge \pi^3.$$
\end{corollary}

The importance of the last corollary will be more clear once all the steps in the proof  of Theorem \ref{main_th}  have been explained. The proof relies on an algorithm which is quite branched, therefore let us briefly outline the global strategy before going into details. 

We will start with a smooth curve $\gamma$, and the first important step is the construction of what we will denote by Procedure 1 and Procedure 2. They are two different maneuvers to continuously modify parts of the curve $\gamma$. During such reshaping of $\gamma$ the area of the set $\Omega_\gamma$ continuously decreases and the elastic energy of $\gamma$ does not increase. Procedure 1 deals with convex arcs, those on which the curvature is positive, Procedure 2 deals with concave arcs, those on which the curvature is negative. 

In particular Procedure 1 can be applied to those convex arcs 
which, when traveling from one end point to another (by end points we mean inflection points), map the normal to $\gamma$ into a subset of $\mathbb{S}^1$ of measure smaller than $\pi$.

\begin{figure}
\begin{minipage}{0.5\textwidth}
\begin{center}
\def\svgwidth{2.in}
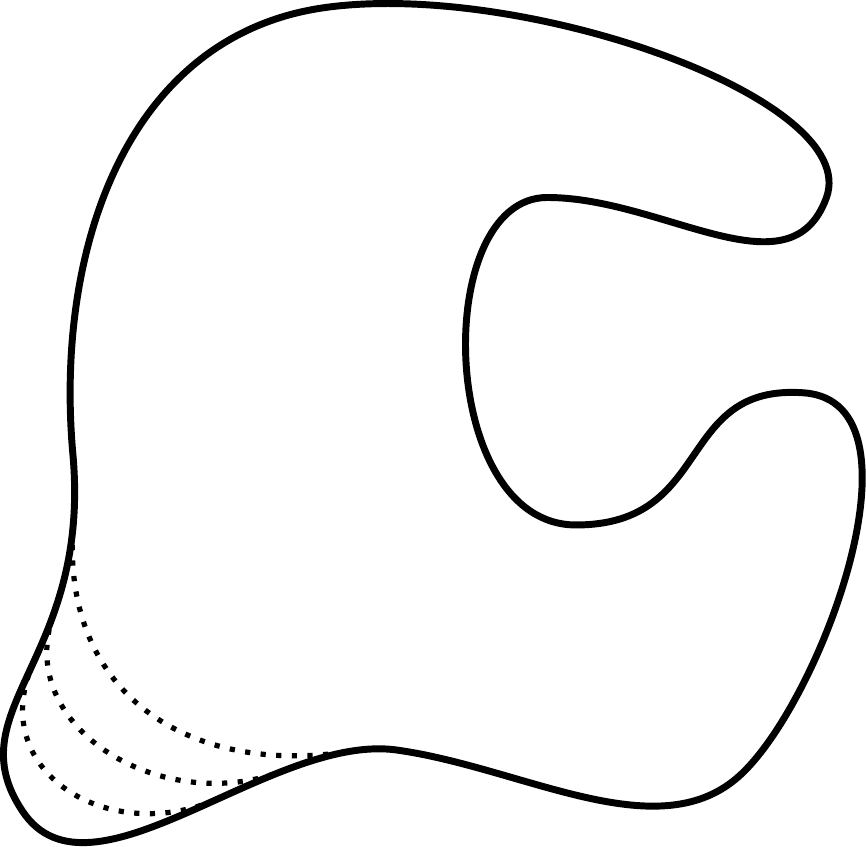\caption{Procedure 1.}\label{fig_pro1}
\end{center}
\end{minipage}%
\begin{minipage}{0.5\textwidth}
\begin{center}
\def\svgwidth{2.in}
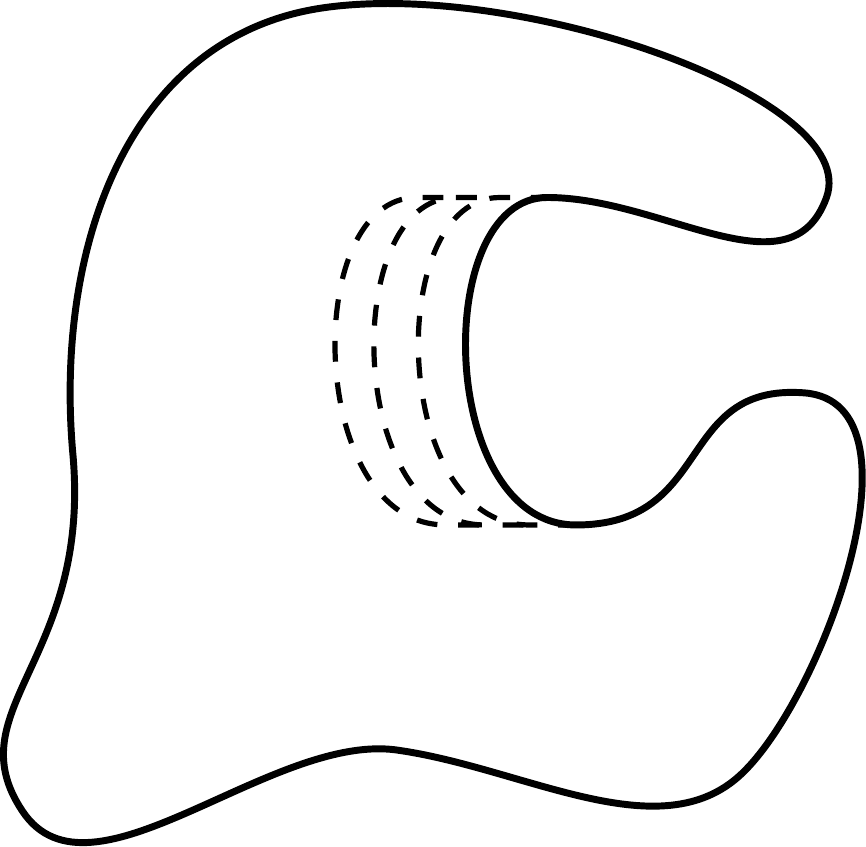\caption{Procedure 2.}\label{fig_pro2}
\end{center}
\end{minipage}
\end{figure}

Procedure 2 on the other hand, can be applied to those concave arcs 
which, when traveling from one end point to another, map the normal to $\gamma$ into a subset of $\mathbb{S}^1$ of measure larger than $\pi$.

Procedure 1 and 2 do not affect the smoothness of $\gamma$. They both modify $\Omega_\gamma$ so that the set shrinks. The perturbation is in the direction of the inner normal (as schematically depicted in Figures \ref{fig_pro1} and \ref{fig_pro2}). The perturbation can be iterated to all arcs fulfilling these requirements, we will apply it arc by arc, as long as the curve remains simple. And here is where things become a bit more complicated. 

In fact, if the curve remains simple along all the reshaping procedures, in the end, after a finite number of iterations,
 we get a curve $\gamma$ which is $C^1$ piecewise $C^2$. Moreover, by construction this curve has all convex arcs mapping the normal to $\gamma$ into a subset of $\mathbb{S}^1$ of measure larger than $\pi$, and all concave arcs mapping the normal to $\gamma$ into a subset of $\mathbb{S}^1$ of measure smaller than $\pi$.
The last paragraph of the proof is devoted to showing that any curve $\gamma$ with these features automatically satisfies the assumptions of Proposition \ref{two_convex}.

On the other hand, by applying Procedure 1 or Procedure 2 the curve $\gamma$ may pinch somewhere, possibly even in infinitely many points. Let us denote by $\Gamma$ the piece of arc that we are reshaping when the first pinching occurs. 
All pinching points belongs obviously to $\Gamma$ (see Figure \ref{fig_first_pinching}). Since, as it will be clear, the arc itself by construction remains simple throughout the deformation, this means that $\Gamma$ is somewhere tangent to some other arc of $\gamma$. We can find on $\Gamma$ two pinching points $p_1$ and $p_2$, where $p_1$ precedes all other pinching points on $\Gamma$, $p_2$ follows all other pinching points on $\Gamma$ (the construction is obviously simplified if pinching occurs just in one point and $p_1$ happens to coincide with $p_2$).
By elementary topological considerations there exists an arc on the curve $\gamma$ with both endpoints on $p_1$, that is simple closed smooth everywhere except in $p_1$ where a cusp appears. At the same time there exists another arc on the curve $\gamma$ with both endpoints on $p_2$, that is simple closed smooth everywhere except in $p_2$ where another cusp appears. We denote these two arcs by $\gamma_1$ and $\gamma_2$, because they are two disjoint closed curves, smooth everywhere but on the cusp. 

\begin{figure}
\def\svgwidth{4.in}
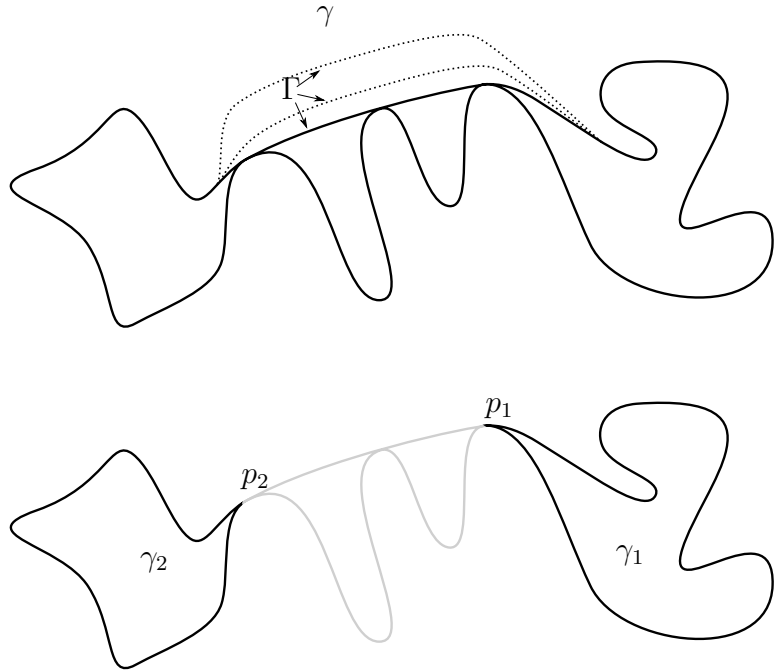\caption{When $\gamma$ pinches somewhere we can identify two arcs which parametrized by arc length are two simple curves $\gamma_1$ and $\gamma_2$ smooth everywhere but on $p_1$ and $p_2$ where each of them has a cusp.}\label{fig_first_pinching}
\end{figure}

The two curves $\gamma_1$ and $\gamma_2$ bound in any case two disjoint sets $\Omega_{\gamma_1}$ and $\Omega_{\gamma_2}$ which overall have smaller area than $\Omega_\gamma$. The curvature is not defined on the cusp but it is an $L^2$ function away from the singularity. So the elastic energy can be computed on both curves and the sum of $E(\gamma_1)$ and $E(\gamma_2)$ is clearly less then $E(\gamma)$. 

At this stage we can still apply Procedure 1 and Procedure 2 to arcs of $\gamma_1$. We just need to stay away from the cusp. And we can still go through a pinching (see Figure \ref{fig_second_pinching}). If a pinching occurs for instance to $\gamma_1$, we can repeat the same argument as before and we can find a pinching point $\tilde p_1$ and arc $\tilde \gamma_1$ on the curve $\gamma_1$ with both endpoints on $\tilde p_1$, that is simple closed smooth everywhere except in $\tilde p_1$ where a cusp appears. Therefore we are back again to a curve $\tilde \gamma_1$ belonging to the same class as $\gamma_1$, i.e.: $C^1$ everywhere but in the cusp, piecewise $C^2$. Both area and elastic energy of $\tilde \gamma_1$ are smaller than those of $\gamma_1$. We rename $\tilde \gamma_1$ as $\gamma_1$ and the procedure continues. After a finite number of iteration the curve $\gamma_1$ has become a curve $C^1$ everywhere but in the cusp, piecewise $C^2$, and by construction has all convex arcs mapping the normal to $\gamma$ into a subset of $\mathbb{S}^1$ of measure larger than $\pi$, and all concave arcs mapping the normal to $\gamma$ into a subset of $\mathbb{S}^1$ of measure smaller than $\pi$. The very same things happens to $\gamma_2$.
\begin{figure}
\def\svgwidth{3.in}
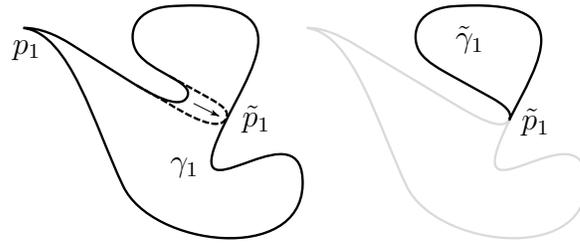\caption{When $\gamma_1$ pinches we can identify on $\gamma_1$ an arc which parametrized by arc length is a simple curve $\tilde\gamma_1$ smooth everywhere but on one point $\tilde p_1$ where it has a cusp.}\label{fig_second_pinching}
\end{figure}
 
In the last paragraph we show that if $\gamma_1$ and $\gamma_2$ have such characteristics, then each one of them has at least one arc holding a convex set, and Corollary \ref{c_two_convex} completes the proof of Theorem \ref{main_th}.

\section{Notation and preliminaries}

In view of Proposition \ref{approx}, from now on, unless otherwise stated, we consider only regular curves $\gamma$ in $\mathbb{R}^2$:
\begin{enumerate}[label=(\emph{\alph*})]
\item which are simple, and for some $L>0$, parametrized by their arc length $s\in[0,L]$;
\item which are of class $C^1([0,L])$ and piecewise $C^2([0,L])$;
\item whose signed curvature changes sign only a finite number of times.
\end{enumerate}

\begin{definition}[Closed smooth curves] We denote by $\mathcal{K}$ the set of curves which satisfy \emph{({\emph a}), ({\emph b}), ({\emph c})}, which are closed (in the sense that $\gamma(0)=\gamma(L)$ and $\gamma'(0)=\gamma'(L)$) and which are positively oriented. 
\end{definition}

\begin{definition}[External cusp] We say that a curve $\gamma$ which satisfies  \emph{({\emph a}), ({\emph b}), ({\emph c})} has an external cusp if  $\gamma(0)=\gamma(L)$, $\gamma'(0)=-\gamma'(L)$, and $\gamma'(L)$ points outside the set $\Omega_\gamma$. We denote by $\mathcal{C}$ the set of such curves upon additionally assuming that they are positively oriented.
\end{definition}

By convention we always parametrize curves $\gamma\in \mathcal{C}$ so that the cusp corresponds to $\gamma(0)=\gamma(L)$. In order to avoid misunderstandings, let us clarify that the elastic energy of a curve $\gamma$ in $\mathcal{C}$ will be
$$\int_0^L k^2(s) ds=\lim_{\varepsilon\downarrow 0}\int_\varepsilon^{L-\varepsilon} k^2(s) ds.$$
Somehow, even if the two endpoints coincide, for what concerns the elastic energy we treat $\gamma\in\mathcal{C}$ as an open curve.

We say that an arc $\wideparen{pq}$ of $\gamma$ is convex if its curvature is nonnegative and not identically vanishing. Conversely we say that $\wideparen{pq}$ of $\gamma$ is concave if the curvature is nonpositive and not identically vanishing. In both cases if the curve belongs to $\mathcal{C}$ then the arc $\wideparen{pq}$ (convex or concave) is assumed not to include the cusp in $\gamma(0)$.

We observe that the arc $\wideparen{pq}\subset\gamma$ can hold a convex set, if and only if the following three conditions hold: 
\begin{enumerate}
\item[(i)] the arc $\wideparen{pq}$ is convex,
\item[(ii)] the total curvature on $\wideparen{pq}$ is $\pi$, i.e.:
$$\int\limits_{\wideparen{pq}}k(s)ds=\pi,$$ 
\item[(iii)] the segment $\overline{pq}$ is included in the closure of $\Omega_\gamma.$
\end{enumerate}

On the curve $\gamma$ an arc $\wideparen{pq}$ contains the arc $\wideparen{p'q'}$, if $p$ precedes $p'$ and $q'$ precedes $q$.
We also say that $\wideparen{pq}$ is a maximal convex arc (maximal concave arc), if there exists no other convex arc (concave arc) $\wideparen{p'q'}$ containing $\wideparen{pq}$ and such that

$$\left|\int\limits_{\wideparen{p'q'}}k(s)ds\right|>\left|\int\limits_{\wideparen{pq}}k(s)ds\right|$$

We finally introduce the oscillation number of the curve $\gamma$, denoted by $\# \gamma$, and defined as the number of maximal disjoint concave arcs that can be found on $\gamma$. If $\#\gamma=0$ then $\Omega_\gamma$ is convex. For $\gamma\in\mathcal{K}$ then $\#\gamma$ is roughly speaking the number of inflections of $\gamma$.
 
\section{proof of Theorem \ref{main_th}}
We want to prove the assertion for a generic curve $\gamma\in\mathcal{K}$.
Since for closed $\gamma\in \mathcal{K}$ such that $\Omega_\gamma$ is convex, Theorem \ref{main_th} follows from Proposition \ref{pr_gage}, we restrict our attention to non convex $\Omega_\gamma$. Therefore we assume that the curve $\gamma$ contains at least one concave arc.

\begin{remark}
If the curve contains just one concave arc, say $\wideparen{p_1p_2}$, then the proof is trivial, since we can immediately apply Proposition \ref{two_convex}. In fact since $$\int\limits_{\wideparen{p_2p_1}}k(s)ds>2\pi$$ there exist two disjoint arcs holding convex sets. In particular we just have to fix two points $q_1$ and $q_2$ on $\gamma$ so that $$\int\limits_{\wideparen{q_1p_1}}k(s)ds=\int\limits_{\wideparen{p_2q_2}}k(s)ds=\pi$$ and then both $\wideparen{q_1p_1}$ and $\wideparen{p_2q_2}$ hold disjoint convex sets. See for instance Figure \ref{fig_ex}(a).
\end{remark}
\begin{figure}
\def\svgwidth{5.in}
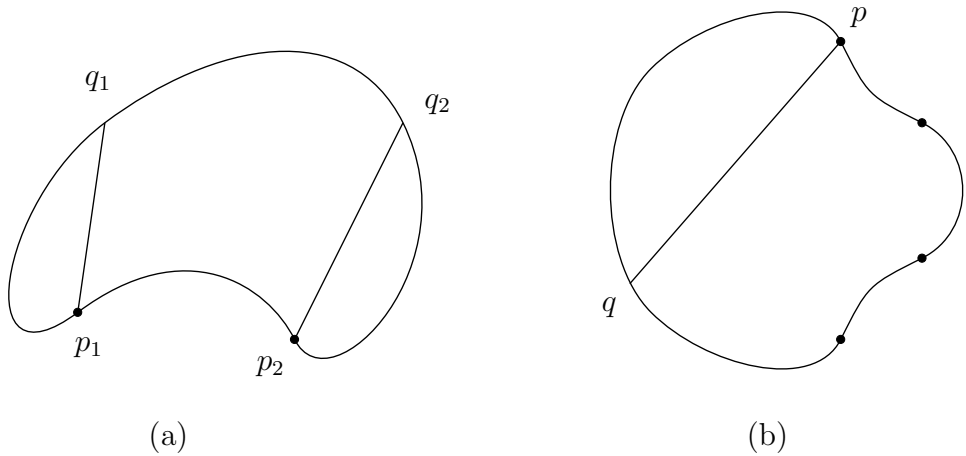\caption{(a) shows a curve $\gamma$ with only one concave arc. Regardless of the shape of the curve, if $p_1$ and $p_2$ are the endpoints of the concave arcs then it is possible to find two points $q_1$ and $q_2$ so that $\wideparen{q_1p_1}$ and $\wideparen{p_2q_2}$ hold two disjoint convex sets. (b) shows a curve $\gamma$ with two concave arcs.  Each ``$\bullet$'' marks a sign change of curvature. One can easily find one arc holding a convex set, for instance the arc $\wideparen{pq}$, however in this case Proposition \ref{two_convex} can not be applied since two disjoint arcs holding  convex set do not exist.} \label{fig_ex}
\end{figure}

It is clear that in general $\gamma$ does not fulfill assumptions in Proposition \ref{two_convex}.
For instance, even if the curve has just two concave arcs it is easy to construct, see for instance Figure \ref{fig_ex}(b), an example in which Proposition \ref{two_convex} can not be applied.

Therefore for general $\gamma$ our aim is to reshape the curve $\gamma$, decreasing its area, not increasing its elastic energy, and eventually being able to use Proposition \ref{two_convex} or Corollary \ref{c_two_convex}. 



\subsection*{Procedure 1 and 2}
There are basically two procedures that we can apply to the curve $\gamma$.
\vspace{.1in}
\paragraph{\it Procedure 1}
The first one consists in finding on the curve $\gamma$ three consecutive maximal arcs (see Figure \ref{fig_3arcs}), namely $\wideparen{p_1 p_2}$, $\wideparen{p_2 p_3}$ and $\wideparen{p_3 p_4}$, 
which are concave, convex and concave, such that

\begin{equation}\label{short_convex}
0<\int\limits_{\wideparen{p_2p_3}}k(s)ds\le \pi.
\end{equation}
\begin{figure}
\def\svgwidth{3.5 in}
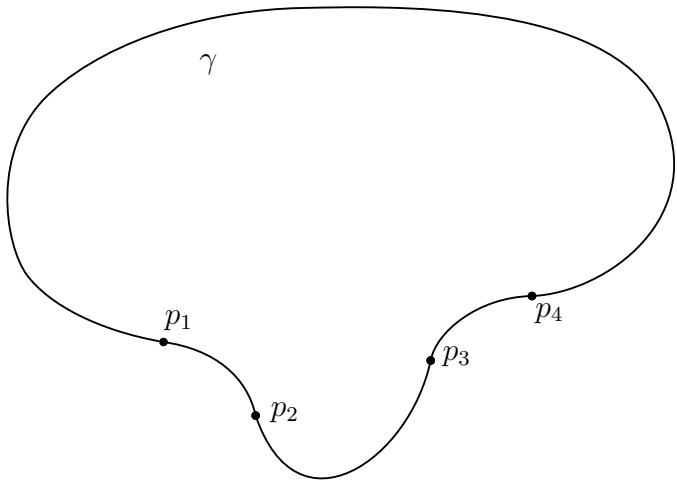\caption{A situation where \eqref{short_convex} holds.}\label{fig_3arcs}
\end{figure}
This allow us to change the shape of the curve $\gamma$ in the following way.

For all positive $\epsilon$ small enough it is possible to find on $\gamma$ two points $p_\epsilon\in\wideparen{p_{1}p_2}$ and $p^\star_\epsilon\in\wideparen{p_2p_3}$ such that the length of $\wideparen{p_\epsilon p^\star_{\epsilon}}$ equals $\epsilon$ and the tangents to $\gamma$ in these two points are parallel to each other. For $\epsilon$ small enough such a couple certainly exists, however it might be not uniquely determined if any of the arcs $\wideparen{p_1p_2} $ and $\wideparen{p_2p_3}$ is not strictly convex. In such a case we can impose the additional assumption that $p_\epsilon$ is the closest point to $p_2$ (in arc length distance) among those fulfilling the previous hypotheses. Thereafter we have uniquely identified  $p_\epsilon$ and $p^\star_\epsilon.$

Now we can consider the translation operator $\mathbb{T_\epsilon}$ acting on $\mathbb{R}^2$ such that $\mathbb{T_\epsilon}q= q-p_\epsilon^\star+p_\epsilon$ for all $q\in\mathbb{R}^2$. The curve $\mathbb{T}_\epsilon\gamma:=\gamma-p_\epsilon^\star+p_\epsilon$, is tangent to $\gamma$ in $p_\epsilon$. If $\epsilon$ is small enough there exists a point $q_\epsilon$ on the arc $\mathbb{T_\epsilon}\, \wideparen{p_2p_3}$ 
and a point $p'_\epsilon$ on the arc $\wideparen{p_3p_4}$ such that the line $r_\epsilon$ passing through these two points is tangent to both $\mathbb{T_\epsilon}\,\wideparen{p_2p_3}$ and $\wideparen{p_3p_4}$.

\begin{figure}
\def\svgwidth{3.5 in}
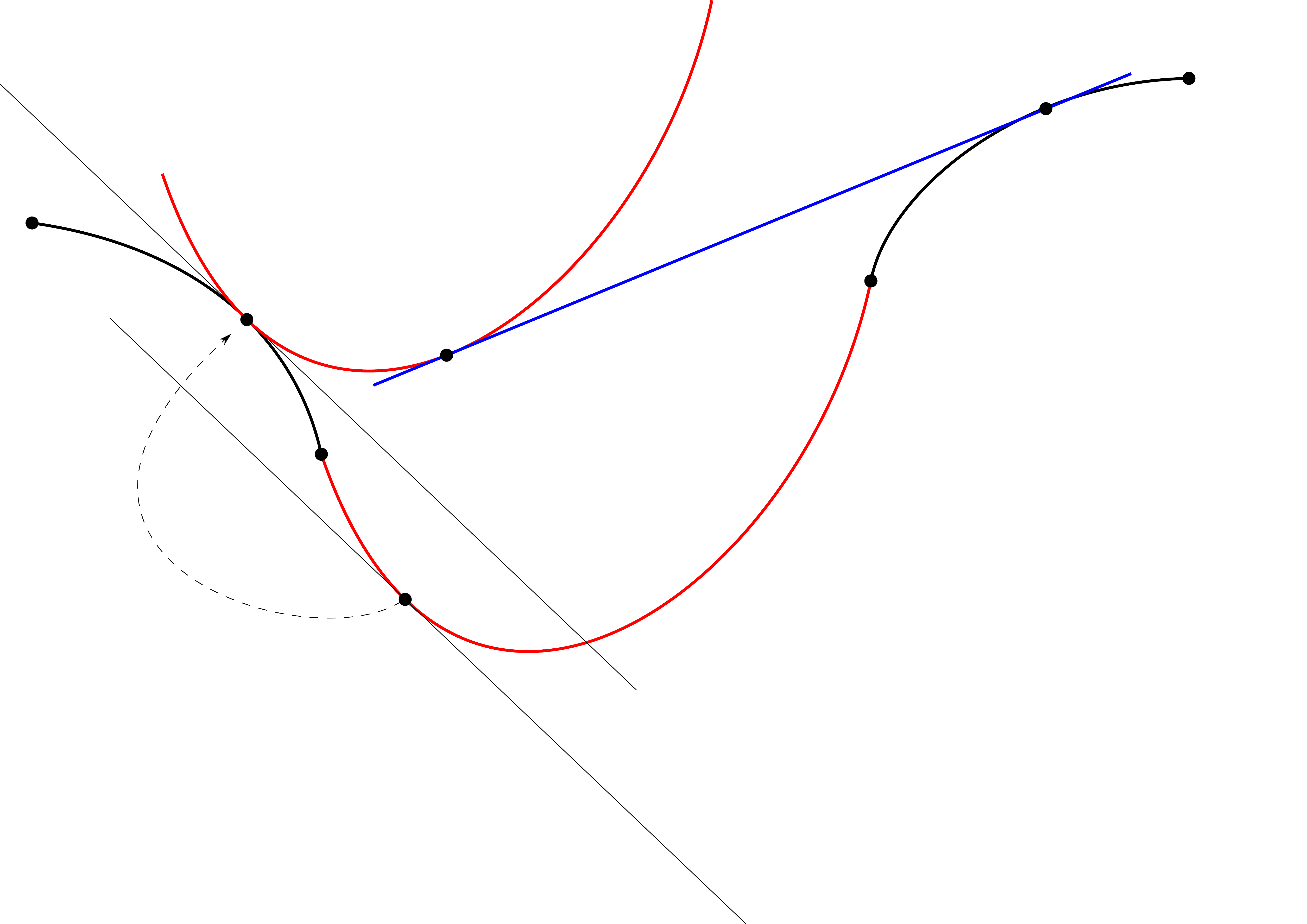\caption{First step in Procedure 1.}
\end{figure}

\begin{figure}
\def\svgwidth{3.5 in}
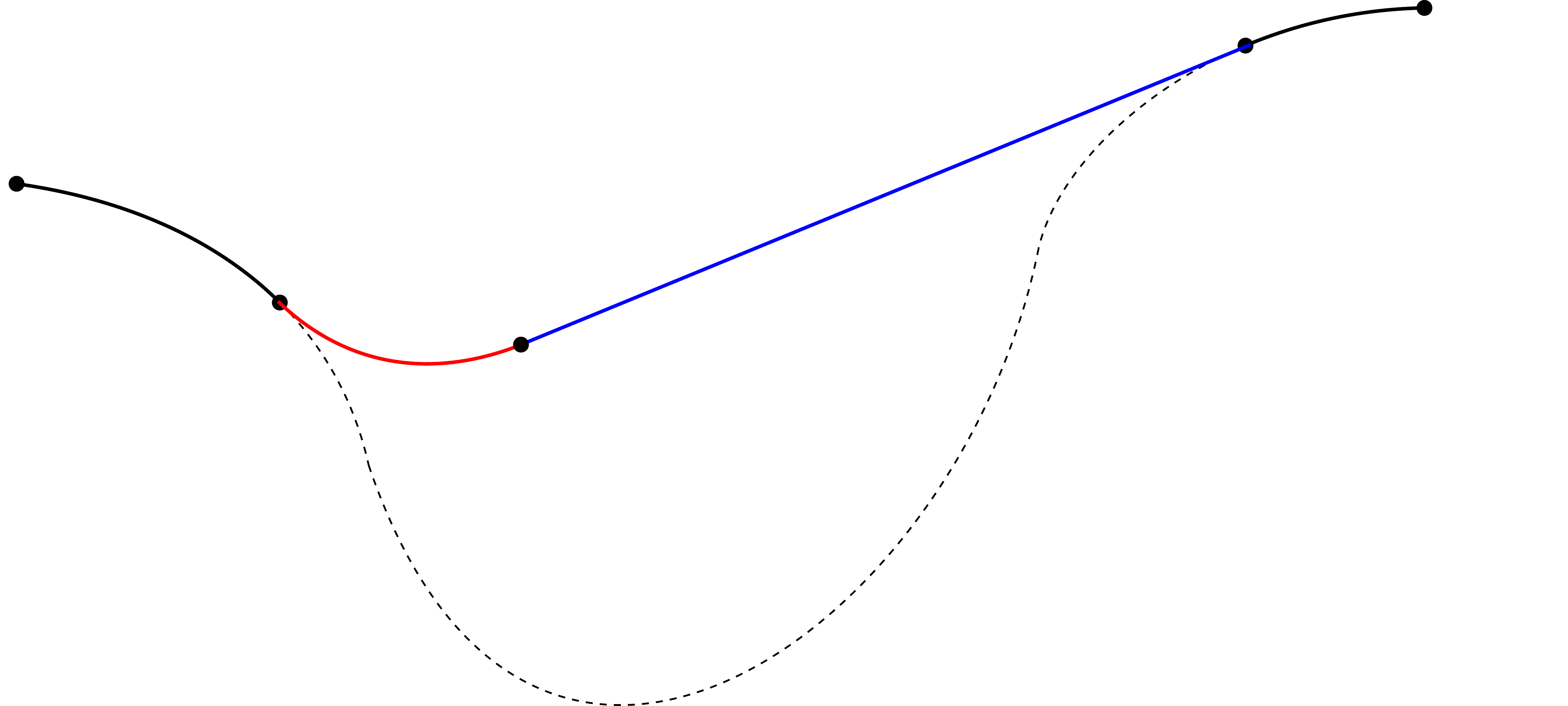\caption{Final shape in Procedure 1.}
\end{figure}

Finally we consider the new curve $\gamma_\epsilon$ which is obtained from $\gamma$ after the arc $\wideparen{p_1p_4}$ is replaced by $\wideparen{p_{1}p_{\epsilon}}\cup\wideparen{p_{\epsilon}q_\epsilon}\cup\overline{q_\epsilon p'_{\epsilon}}\cup \wideparen{p'_{\epsilon}p_{4}}$. Here the arc $\wideparen{p_{\epsilon}q_\epsilon}$ is meant to be an arc of $\mathbb{T}_\epsilon \gamma$, while $\wideparen{p_{1}p_{\epsilon}}$ and $\wideparen{p'_{\epsilon}p_4}$ are arcs of $\gamma$.

Clearly for all positive $\epsilon$ small enough we have $A(\gamma_\epsilon)< A(\gamma)$ and $E(\gamma_\epsilon) \le E(\gamma)$. 
For the construction to work 
$\mathbb{T}_\epsilon\wideparen{p_2p_3}$ has to belong to the closure of $\Omega_\gamma$ at least for small enough $\epsilon>0$, and this is true if and only if $\eqref{short_convex}$ holds.

Let us denote by $\bar\epsilon$ the supremum of all $\epsilon$ such that the construction can be worked out and the curve $\gamma_\epsilon$ is simple. At least one of the following facts certainly occurs:
\begin{enumerate}[label=\textbf{F.\arabic*}]
\item 
the point $p_{\bar\epsilon}$ coincides with $p_{1}$,
\item the point $q_{\bar\epsilon}$ coincides with $p_{\bar\epsilon}$,
\item the point $p'_{\bar\epsilon}$ coincides with $p_{4}$,
\item the curve $\gamma_{\bar\epsilon}$ is not simple and pinches somewhere in between $p_{\bar\epsilon}$ and $p'_{\bar\epsilon}$.\label{pinching}
\end{enumerate}

\vspace{.1in}
\paragraph{\it Procedure 2} The second procedure consists in finding 
a concave arc $\wideparen{p_1 p_2}$ on $\gamma$ such that (see Figure \ref{fig_concave1})

\begin{equation}\label{large_concave}
\int\limits_{\wideparen{p_1p_2}}k(s)ds\le -\pi.
\end{equation}

\begin{figure}
\def\svgwidth{3.in}
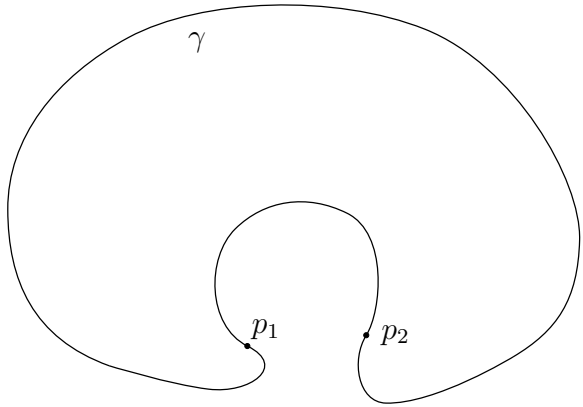\caption{A situation where \eqref{large_concave} holds.}\label{fig_concave1}
\end{figure}
This allow us to change the shape of the curve $\gamma$ in the following way.

Condition \eqref{large_concave} implies that there exist a point $p'_1$ on $\wideparen{p_1p_2}$, such that 
$$\int\limits_{\wideparen{p_1p'_1}}k(s)ds= -\pi,$$ so that the tangents to $\gamma$ on $p_1$ and $p'_1$ are parallel to each other. Then we consider a positive $\epsilon$ small enough and we simply translate by $\epsilon$ the arc $\wideparen{p_1p'_1}$ along the direction of these parallel lines, inward with respect to $\Omega_\gamma$ (see Figure \ref{fig_concave2}). To be more precise, let us denote by ${\bf t}_1$ the tangent vector in $p_1$. We construct a new curve $\gamma_\epsilon$ by replacing the arc $\wideparen{p_1p'_1}$ with the union of:
\begin{itemize}
\item the segment of endpoints $p_1+\epsilon {\bf t}_1$ and $p_1$;
\item the arc $\wideparen{p_1p'_1}+\epsilon {\bf t}_1$;
\item the segment of endpoints $p'_1$ and $p'_1+\epsilon {\bf t}_1$.
\end{itemize}

Figure \ref{fig_concave_end} depicts the shape of $\gamma_\epsilon$. Clearly for all positive $\epsilon$ small enough we have $A(\gamma_\epsilon)< A(\gamma)$ and $E(\gamma_\epsilon) = E(\gamma)$. For the construction to work, condition \eqref{large_concave} is clearly necessary.
Let us denote by $\bar\epsilon$ the supremum of all $\epsilon$ such that the procedure can be carried on and the curve $\gamma_\epsilon$ is simple. Clearly by construction $\gamma_{\bar\epsilon}$ is not simple and in particular the curve pinches somewhere on the arc of end points $p_1+\bar\epsilon {\bf t}_1$ and $p'_1+\bar\epsilon {\bf t}_1$.\\
We emphasize that the very same construction can be worked out using the endpoint $p_2$ and the tangent vector ${\bf t}_2$ instead of the end point $p_1$ and tangent vector ${\bf t}_1$ (see also Figures \ref{fig_concave_other1} and \ref{fig_concave_otherend}). Having this in mind, in the forthcoming paragraphs when necessary we specify whether we apply Procedure 2 by using point $p_1$ or $p_2$.
 

\begin{figure}
\def\svgwidth{3.in}
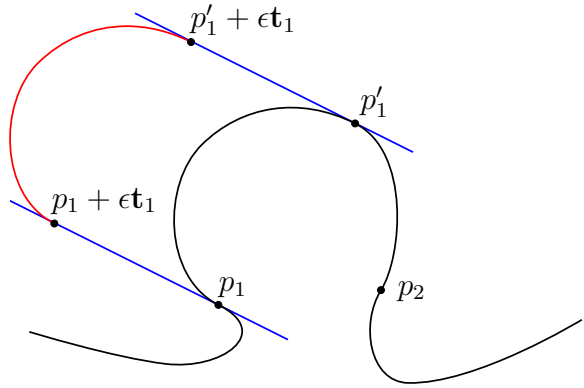\caption{First step in Procedure 2.}\label{fig_concave2}
\end{figure}

\begin{figure}
\def\svgwidth{3.in}
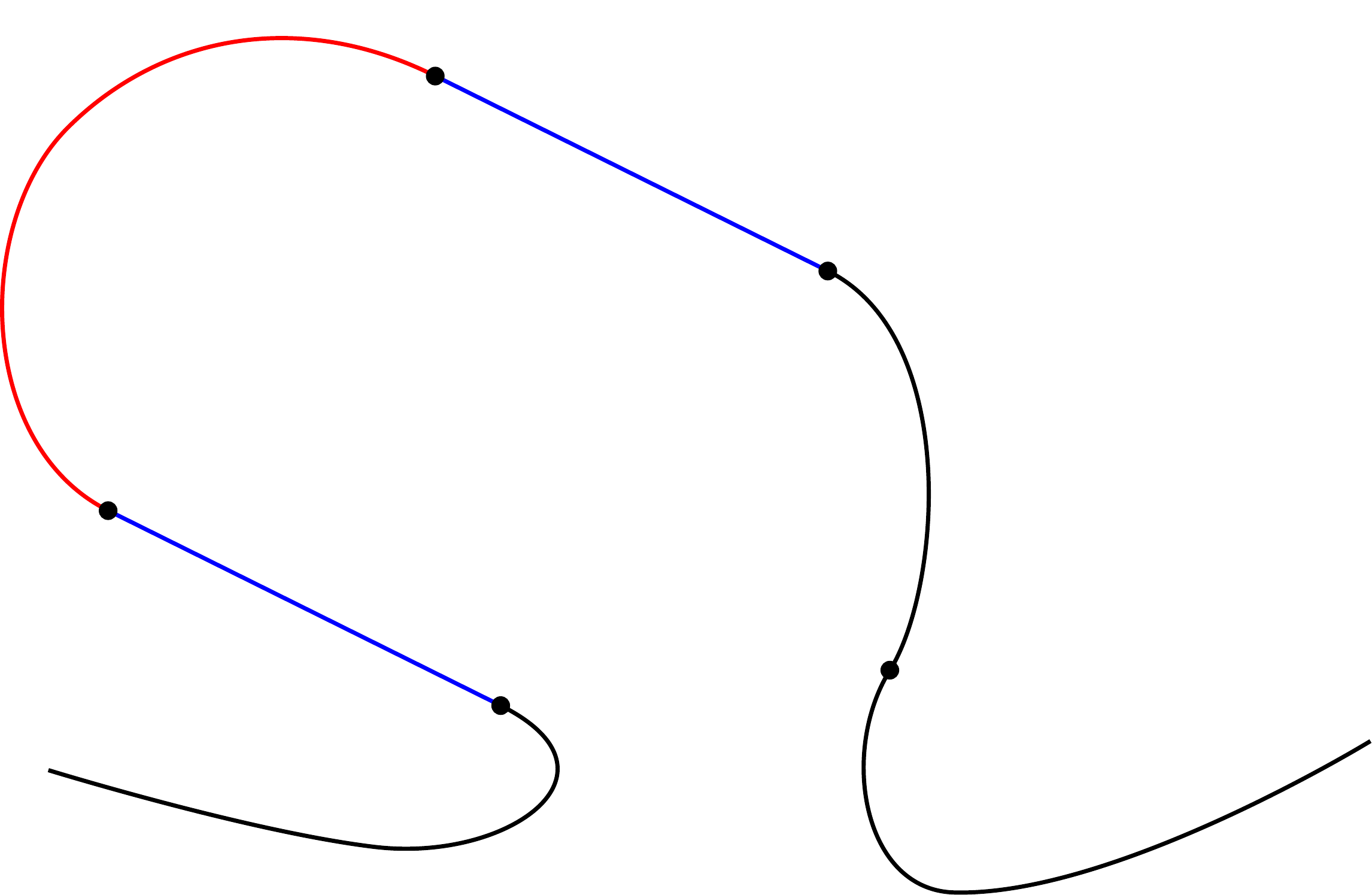\caption{Final shape in Procedure 2.}\label{fig_concave_end}
\end{figure}

\begin{figure}
\def\svgwidth{3.in}
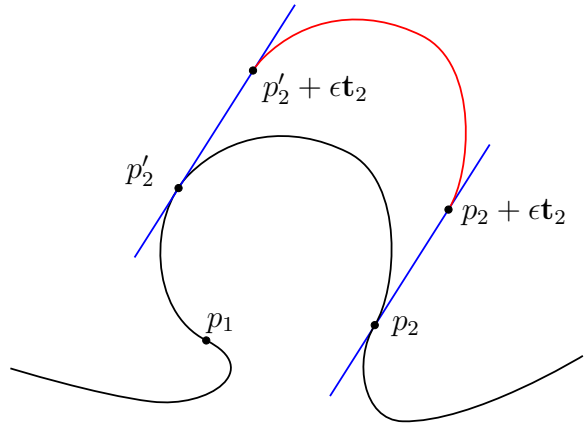\caption{Alternative first step in Procedure 2.}\label{fig_concave_other1}
\end{figure}

\begin{figure}
\def\svgwidth{3.in}
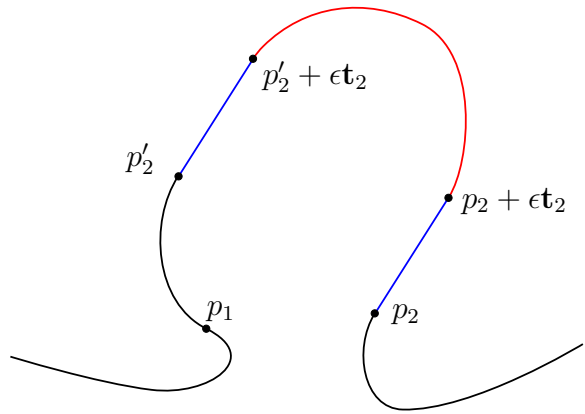\caption{Alternative final shape in Procedure 2.}\label{fig_concave_otherend}
\end{figure}

\subsection*{The Algorithm} In view of the previous subsection we describe here the algorithm which yields to the final shape.

\vspace{.1in}
\paragraph{\it A) Smooth case}
Once we have defined Procedures 1 and 2 we can start reshaping our initial curve $\gamma\in\mathcal{K}$. We start applying, if possible, Procedure 1 to some convex arc. In case condition \ref{pinching} does not occur, then $\#\gamma_{\bar\epsilon}<\#\gamma$. Obviously $A(\gamma_{\bar\epsilon})<A(\gamma)$ and $E(\gamma_{\bar\epsilon})\le E(\gamma).$ We rename $\gamma_{\bar\epsilon}$, and for simplicity we denote it by $\gamma$. We go on repeating Procedure 1, as long as condition \ref{pinching} does not occur, and we can find convex arcs to which Procedure 1 can be applied. Observe that this stage ends after a finite number of iterations. \\

Let us assume for the moment that we iterate the process and we end without pinching, then we still have a curve in $\mathcal{K}$ and we can look for the possibility to apply Procedure 2. If this is possible the curve undergoes pinching. If it is not possible to apply Procedure 2 then $\gamma$ is a curve in $\mathcal{K}$ such that all maximal convex arcs $\wideparen{pq}$ satisfy
\begin{equation}\label{eq_KP1}\int\limits_{\wideparen{pq}} k(s)ds> \pi,
\end{equation}
and all maximal concave arcs $\wideparen{pq}$ satisfy
\begin{equation}\label{eq_KP2}0>\int\limits_{\wideparen{pq}} k(s)ds> -\pi.
\end{equation}

\begin{definition}[Class $\mathcal{K}_\pi$]\label{def_KP} We say that a curve in $\mathcal{K}$ belongs to the class $\mathcal{K}_\pi$ if all maximal convex arcs $\wideparen{pq}$ satisfy \eqref{eq_KP1} and all maximal concave arcs $\wideparen{pq}$ satisfy \eqref{eq_KP2}. 
\end{definition}

The class $\mathcal{K}_\pi$ contains obviously convex sets, but also non trivial sets like those in Figure \ref{fig_final_smooth}. In principle there is no bound on $\#\gamma$ if $\gamma\in\mathcal{K}_\pi$. 

\begin{figure}
\def\svgwidth{6.in}
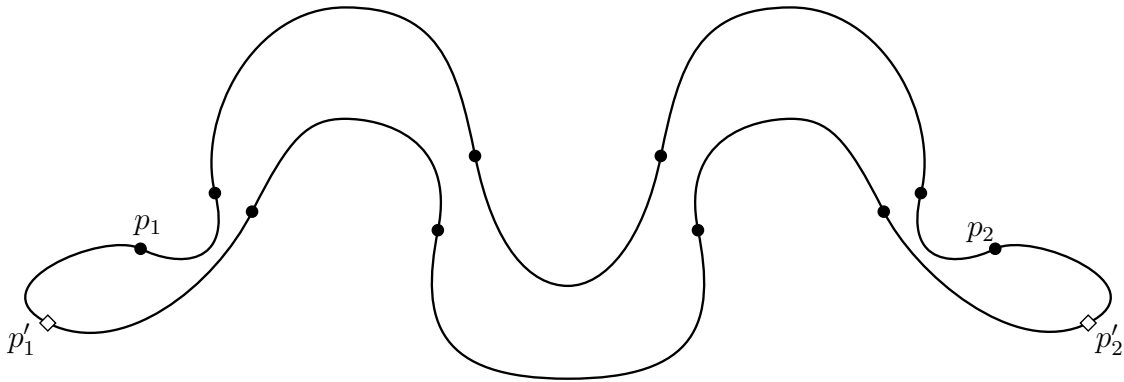\caption{This figure shows a curve in $\mathcal{K}_\pi$. We marked with ``$\bullet$'' all points where curvature changes sign. We also marked two arcs $\wideparen{p_1p'_1}$ and $\wideparen{p'_2p_2}$ which hold two disjoint convex sets. Observe that any other arc holding a convex set has non empty intersection with one of this two.}\label{fig_final_smooth}
\end{figure}

\vspace{.1in}
\paragraph{\it B.1) Non smooth case: the first pinching.}
What happens if pinching occurs? This might happen if in the previous paragraph we are able to apply Procedure 2 or, if Procedure 1 leads to condition \ref{pinching}. 

In any case, we can assume that we start from a curve $\gamma$ in $\mathcal{K}$ and we end up with the curve $\gamma_{\bar\epsilon}$ which pinches somewhere. The simplest case is that the pinching occurs just in one point. In this case we can split the curve into two curves $\gamma_1$ and $\gamma_2$ belonging to the class $\mathcal{C}$. 
Indeed, the fact that the curve $\gamma_{\bar\epsilon}$ pinches in one point means that there exists a unique couple $(s_1,s'_1)$, with $0\le s_1 < s'_1$ such that  
$\gamma_{\bar\epsilon}(s_1)=\gamma_{\bar\epsilon}(s'_1)$. By construction the arc of curve 
in between $\gamma_{\bar\epsilon}(s_1)$ and $\gamma_{\bar\epsilon}(s'_1)$ 
can be parametrized by a curve $\gamma_1$ which belongs to $\mathcal{C}$. The arc of curve between 
$\gamma_{\bar\epsilon}(s')$ and $\gamma_{\bar\epsilon}(s)$ 
can be as well parametrized by a curve $\gamma_2$ which belongs to $\mathcal{C}$. By trivial continuity argument we have
$$A(\gamma_1)+A(\gamma_2)< A(\gamma)\quad\mbox{and}\quad E(\gamma_1)+E(\gamma_2)\le E(\gamma).$$

In case the pinching of $\gamma_{\bar\epsilon}$ occurs in more then one point, as in Figure \ref{fig_first_pinching}, then by an appropriate parametrization we can find two couples $(s_1,s'_1)$ and $(s_2,s'_2)$ such that
\begin{itemize}
\item $s_1 < s_2 < s'_2< s'_1 $, 
\item $\gamma_{\bar\epsilon}(s_1)=\gamma_{\bar\epsilon}(s'_1)$ and $\gamma_{\bar\epsilon}(s_2)=\gamma_{\bar\epsilon}(s'_2)$,
\item for all $0\le \bar s < \bar s'$ such that $\gamma_{\bar\epsilon}(\bar s)=\gamma_{\bar\epsilon}(\bar s')$ we have $s_1 \le \bar s \le s_2$ and $s'_1 \ge \bar s' \ge s'_2$.
\end{itemize}

This means that we are choosing $p_1\equiv\gamma_{\bar\epsilon}(s_1)$ and $p_2\equiv\gamma_{\bar\epsilon}(s_2)$ so that by traveling along the curve between $s_1$ and $s_2$ we pass through all pinching points of $\gamma_{\bar\epsilon}$ one and only one time. To a certain extent $p_1$ and $p_2$ correspond to the ``first'' and the ``last'' pinching point of $\gamma_{\bar\epsilon}$.

Once again by construction the arc of curve
in between $\gamma_{\bar\epsilon}(s'_1)$ and $\gamma_{\bar\epsilon}(s_1)$ 
can be parametrized by a curve $\gamma_1$ which belongs to $\mathcal{C}$. The arc of curve in between $\gamma_{\bar\epsilon}(s_2)$ and $\gamma_{\bar\epsilon}(s'_2)$ 
can be as well parametrized by a curve $\gamma_2$ which belongs to $\mathcal{C}$. 

By trivial continuity argument we have
\begin{equation}\label{getting_better}
A(\gamma_1)+A(\gamma_2)< A(\gamma)\quad\mbox{and}\quad E(\gamma_1)+E(\gamma_2)\le E(\gamma).
\end{equation}

\vspace{.1in}
\paragraph{\it B.2) Non smooth case: after the first pinching}
After the first pinching occurs we have to deal with two curves $\gamma_1$ and $\gamma_2$ both belonging to the class $\mathcal{C}$. Let us focus on $\gamma_1$. All what we say for $\gamma_1$ can be repeated for $\gamma_2$. We start by observing that Procedure 1 can be applied to $\gamma_1$, in the very same way we already did to $\gamma$, 
whenever we find three maximal arcs $\wideparen{p_1 p_2}$, $\wideparen{p_2 p_3}$ and $\wideparen{p_3 p_4}$, 
respectively alternately concave, convex and concave, such that \eqref{short_convex} holds true, with the additional assumption that the cusp is neither in $p_2$ nor in $p_3$. Assuming we can apply Procedure 1 to $\gamma_1$, then, if condition \ref{pinching} does not occur, we end up with the curve $\gamma_{1\bar\epsilon}$ which is still in $\mathcal{C}$ and after renaming the curve $\gamma_{1\bar\epsilon}$ as $\gamma_1$, inequalities \eqref{getting_better} still hold true.
Therefore our strategy is to iterate Procedure 1 to both $\gamma_1$ and $\gamma_2$ as long as \ref{pinching} does not occur and as long as there exist arcs fulfilling the necessary requirements.

Moreover we observe that Procedure 2 can be also applied to $\gamma_1$ (and/or $\gamma_2$) if there exists a maximal concave arc $\wideparen{p_1p_2}$ such that \eqref{large_concave} holds true. This time we just have to be careful that if $p_1$ correspond to the cusp, we will perform the construction of Procedure 2 by using the point $p_2$ and viceversa. 
If some arc satisfies the conditions to apply Procedure 2, then the curve will necessarily undergo another pinching.
\\

\vspace{.1in}
\paragraph{\it B.3) Non smooth case: dealing with subsequent pinchings}
What happens if pinching occurs again?\\
Assume that we apply Procedure 2 to $\gamma_1\in \mathcal{C}$ as in Figure \ref{fig_second_pinching}, or that condition \ref{pinching} occurs after applying Procedure 1 to $\gamma_1\in\mathcal{C}$, 
in both cases we have a curve $\gamma_{1\bar\epsilon}$ that in general pinches in more then one point.  We can find two couples $(s_1,s'_1)$ and $(s_2,s'_2)$ such that
\begin{itemize}
\item $s_1 \le s_2 < s'_2\le s'_1 $, 
\item $\gamma_{1\bar\epsilon}(s_1)=\gamma_{1\bar\epsilon}(s'_1)$ and $\gamma_{1\bar\epsilon}(s_2)=\gamma_{1\bar\epsilon}(s'_2)$,
\item for all $0< \bar s < \bar s'$ such that $\gamma_{1\bar\epsilon}(\bar s)=\gamma_{1\bar\epsilon}(\bar s')$ we have $s_1 \le \bar s \le s_2$ and $s'_1 \ge \bar s' \ge s'_2$.
\end{itemize}

Remember that $\gamma_{1\bar\epsilon}(0)=\gamma_{1\bar\epsilon}(L)$ is the cusp of $\gamma_{1\bar\epsilon}.$ 

This means that, by traveling along the curve between $\gamma_{1\bar\epsilon}(s_1)$ and $\gamma_{1\bar\epsilon}(s_2)$ we pass through all pinching points of $\gamma_{\bar\epsilon}$ one and only one time, and to a certain extent the point $\tilde p\equiv\gamma_{1\bar\epsilon}(s_2)=\gamma_{1\bar\epsilon}(s'_2)$ corresponds to the ``last'' pinching point of $\gamma_{1\bar\epsilon}$.

By construction the arc of curve
in between $\gamma_{1\bar\epsilon}(s_2)$ and $\gamma_{1\bar\epsilon}(s'_2)$ 
can be parametrized by a curve which belongs to $\mathcal{C}$. We rename such a curve $\gamma_1$ and we observe that once again inequalities \eqref{getting_better} hold true.

\vspace{.1in}
\paragraph{\it B.4) Non smooth case: the final shape}
Now it is clear that, even in case that pinching occurs, we can apply Procedure 1 and Procedure 2 to the curves $\gamma_1$ and $\gamma_2$ and iterate the process (a finite number of times) until we end up with two curves in the class $\mathcal{C}$, such that on each curve all maximal convex arcs $\wideparen{pq}$ not adjacent to the cusp satisfy
\begin{equation}\label{eq_CP1}\int\limits_{\wideparen{pq}} k(s)ds> \pi,
\end{equation}
and all maximal concave arcs $\wideparen{pq}$ satisfy
\begin{equation}\label{eq_CP2}0>\int\limits_{\wideparen{pq}} k(s)ds> -\pi.
\end{equation}

\begin{definition}[Class $\mathcal{C}_\pi$]\label{def_CP} We say that a curve in $\mathcal{C}$ belongs to the class $\mathcal{C}_\pi$ if all maximal convex arcs $\wideparen{pq}$ not adjacent to the cusp satisfy \eqref{eq_CP1} and all maximal concave arcs $\wideparen{pq}$ satisfy \eqref{eq_CP2}. 
\end{definition}


\subsection*{Some properties of $\mathcal{K}_\pi$ and $\mathcal{C}_\pi$}
Now that we introduced the sets $\mathcal{K}_\pi$ and $\mathcal{C}_\pi$, we observe the following facts.

We remind that, if $\wideparen{pq}$ is any arc of $\gamma$, by total curvature of $\wideparen{pq}$ we mean $$\int_{\wideparen{pq}}k(s)ds,$$
and the above integral represents the signed difference in angle between the tangent vector in $q$ and the tangent vector in $p$, also called rotation angle. If the curve $\gamma$ belongs to $\mathcal{C}_\pi$ and the cusp is somewhere in between $p$ and $q$ then the total curvature can be still defined as 
\begin{equation}\label{rot_c}
\int_{\wideparen{pq}}k(s)ds+\pi.
\end{equation}
For a piecewise smooth curve indeed the rotation angle is defined to be the sum of the changes in the angles along each
smooth arc plus the sum of the jump angles in singular points. The changes in the angles along a
smooth arc is the curvature integral along the arc. The jump angle at a singular point $\gamma(s_0)$ is defined to be the angle
from the incoming tangent vector $\gamma'(s_0^-)$ to the outgoing tangent vector $\gamma'(s_0^+)$. For curves in $\mathcal{C}_\pi$ the rotation angle in the cusp is $\pi$. This clarifies the definition in \eqref{rot_c}.

In view of the above considerations, a straightforward consequence of Definitions \ref{def_KP} and \ref{def_CP} is the following Lemma.

\begin{lemma}\label{lem_total} If $\gamma$ is a curve in $\mathcal{K}_\pi$ or in $\mathcal{C}_\pi$, and $\wideparen{pq}$ is any arc of $\gamma$, then the total
curvature of $\wideparen{pq}$ is always strictly grater then $-\pi$.
\end{lemma}

We have the following Lemma.
\begin{lemma}\label{lem_rotation}
Let $\wideparen{pq}$ be an arc of a curve $\gamma\in\mathcal{K}_\pi$ (or $\gamma\in\mathcal{C}_\pi$ as long as $p$ is not the cusp). Let us denote by $r$ the tangent line to $\gamma$ in $p$ and assume that $q$ also belongs to $r$. The line $r$ splits the plane in two and we assume that all points $w\in\gamma$ which follow $p$ and precedes $q$ are on the same side of the plane, namely the opposite one with respect to where the normal to $\gamma$ in $p$ is pointing. Then, the vector $\frac{q-p}{|q-p|}$ is the tangent vector to $\gamma$ in $p$ (namely the situation in Figure \ref{fig_rotation}(b) is infeasible).
\end{lemma}
\begin{proof}
The proof is the same for curves in $\mathcal{K}_\pi$ and  $\mathcal{C}_\pi$. For simplicity let us assume $\gamma\in\mathcal{K}_\pi$ and therefore $\wideparen{pq}$ is a smooth arc. The situation is depicted in Figure \ref{fig_rotation} and basically there are two different configurations (a) and (b). In both (a) and (b) we consider the piecewise smooth simple closed curve $\psi$ given by the union of the arc $\wideparen{pq}$ and the segment $\overline{pq}$. The orientation of the curve $\psi$ is the one induced by the orientation of $\gamma$. In (a) it is counterclockwise oriented. In (b) it is clockwise oriented. From the well known rotation index theorem \cite{Hopf1935}, the rotation angle of the curve $\psi$ in (a) is $2\pi$ while in (b) it is $-2\pi$. It is easy to prove that in (a) the total curvature of the arc $\wideparen{pq}$ is positive indeed the jump angle in $p$ is $\pi$ and the jump angle in $q$ is less or equal then $\pi$. On the other hand in (b) the total curvature of the arc $\wideparen{pq}$ is negative and in particular less or equal than $-\pi$. In fact the jump angle in $q$ is more then $-\pi$ and in $p$ there is no jump since $p$ is a regular point for $\psi$. Therefore according to Lemma \ref{lem_total} the configuration in Figure \ref{fig_rotation}(b) is impossible. 

\begin{figure}
\def\svgwidth{4.in}
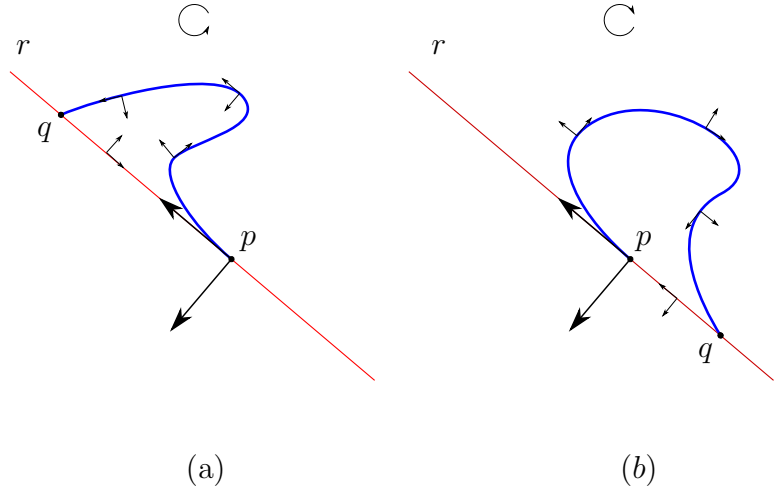\caption{This figure shows the two possible configurations. In (a) the point $q$ follows $p$ on $r$. In (b) the point $q$ precedes $p$ on $r$. Situation (b) is impossible for an arc of a curve $\gamma$ in $\mathcal{K}_\pi$.}\label{fig_rotation}
\end{figure}
\end{proof}

We introduce the following definition.
\begin{definition}
Let $\gamma$ be a curve in $\mathcal{K}_\pi$ or $\mathcal{C}_\pi$, and let $\wideparen{pq}$ be a maximal concave arc. Assume that $p'$ and $q'$ are such that
\begin{itemize}
\item $\wideparen{p'p}$ is convex and the total curvature is $\pi$;
\item $\wideparen{qq'}$ is convex and the total curvature is $\pi$;
\item $\wideparen{qq'}$ intersects $\overline{pp'}$.
\end{itemize}
We say in this case that the arc $\wideparen{qq'}$ is nested into the arc $\wideparen{p'p}$ (see also Figure \ref{fig_nested} (a)).
\end{definition} 
Now we can state a fundamental property of nested arcs.
\begin{proposition}\label{prop_nested}
For a curve $\gamma$ in $\mathcal{K}_\pi$ or $\mathcal{C}_\pi$, let the arc $\wideparen{qq'}$ be nested into the arc $\wideparen{p'p}$. Then $\wideparen{qq'}$ holds a convex sets.
\end{proposition}
\begin{proof}
Arguing by contradiction we assume that the curve $\gamma$ crosses the segment $\overline{qq'}$. If so, by continuity argument there certainly exists a point $\tilde q$ on the arc $\wideparen{qq'}$ such that the curve $\gamma$ is tangent to the segment $\overline{\tilde qq'}$ but does not cross such segment (obviously the end points $\tilde q$ and $q'$ do not count). Let us call $q^\star$ such a tangent point. If the point is not unique, we denote by $q^\star$ the closest one to $\tilde q$. Since the open set bounded by $\overline{\tilde qq'}$ and
$q\hskip-6pt\wideparen{\,{\tilde{\ }}{q'}}$ is a subset of $\Omega_\gamma$, the tangent vector to $\gamma$ at $q^\star$ points in the direction of $\tilde q$.

Let us also denote by $r$ the line passing through $\tilde q$ and $q'$. Thereafter we consider the arc $\wideparen{q^\star p'}$. Following the orientation of $\gamma$ and starting from $q^\star$, the first time such an arc crosses the line $r$ (see Figure \ref{fig_nested}(b)) it violates the conclusion of Lemma \ref{lem_rotation}. Therefore we get a contradiction.
\begin{figure}
\def\svgwidth{6.in}
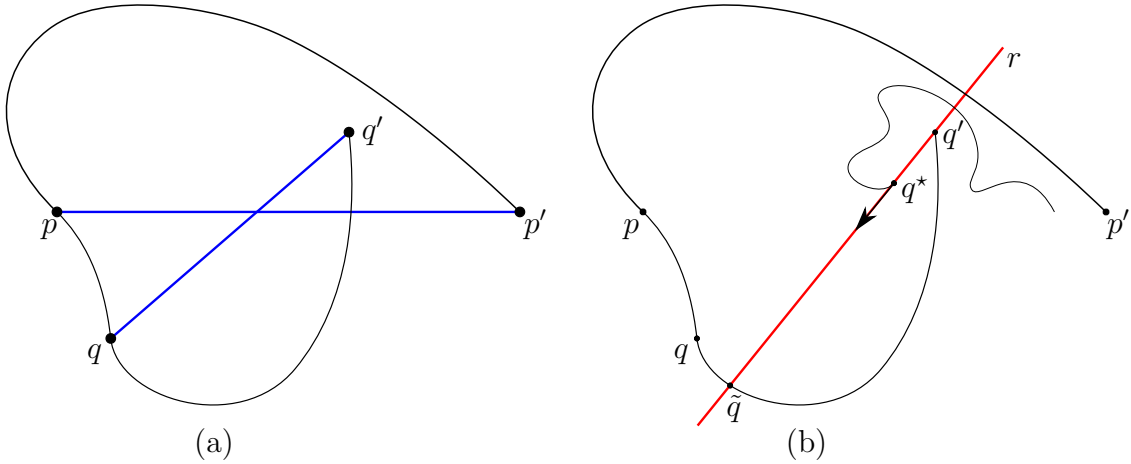\caption{Figure (a) represents the arc $\wideparen{qq'}$ nested in $\wideparen{p'p}$. Figure (b) shows that if $\gamma$ crosses $\overline{qq'}$ then the curve would violate the conclusion of Lemma \ref{lem_rotation}.}\label{fig_nested}
\end{figure}
\end{proof}

\subsection*{End of the proof for final shape in $\mathcal{K}_\pi$}
In this section we show that for curves in $\mathcal{K}_\pi$ Theorem \ref{main_th} is true. This concludes the proof of the Theorem when during the reduction algorithm no pinching occurs to $\gamma$.\\
Let $\gamma\in \mathcal{K}_\pi$. Given two distinct points $h$ and $f$ on $\gamma$ we say that $\overline{hf}$ is a chord for $\gamma$ if, with the exception of the end points, it is contained in $\Omega_\gamma$. 

Given a maximal convex arc $\wideparen{pq}$ on $\gamma$ we denote by $p^\sharp$ and $q^\sharp$ two points such that 
$$\int\limits_{\wideparen{pp^\sharp}} k(s)ds=\pi,$$ and
$$\int\limits_{\wideparen{q^\sharp q}} k(s)ds=\pi.$$
In case $p^\sharp$ is not uniquely determined (this might happen possibly if somewhere $\gamma$ contains a segment and $p^\sharp$ belongs to such a segment), we choose the point which also minimizes the arc length distance to $p$. In case $q^\sharp$ is not uniquely determined, we choose the point which also minimizes the arc length distance to $q$. \\ 

Example in Figure \ref{fig_ex}(a) shows that if $\#\gamma=1$ then Theorem \ref{main_th} is true. Therefore we assume that $\#\gamma\ge 2$. Moreover we can assume that there exists at least one maximal convex arc $\wideparen{q_1p_1}$ such that none of the arcs $\wideparen{q'p'}\subset\wideparen{q_1p_1}$ holds a convex set, otherwise the proof is complete. We say in this case that arc $\wideparen{q_1p_1}$ is \emph{void}. If $\wideparen{q_1p_1}$ is void then $\gamma$ crosses the segment $\overline{p_1p_1^\sharp}$ somewhere in between the two endpoints. There exists therefore $p_1^\dagger$ on $\overline{p_1p_1^\sharp}$ such that $\overline{p_1p_1^\dagger}$ is a chord for $\gamma$. 
In the same way we can find $q_1^\dagger$ on $\overline{q_1q_1^\sharp}$ such that $\overline{q_1q_1^\dagger}$ is a chord for $\gamma$. See Figure \ref{fig_sharp} 

\begin{figure}
\def\svgwidth{4.in}
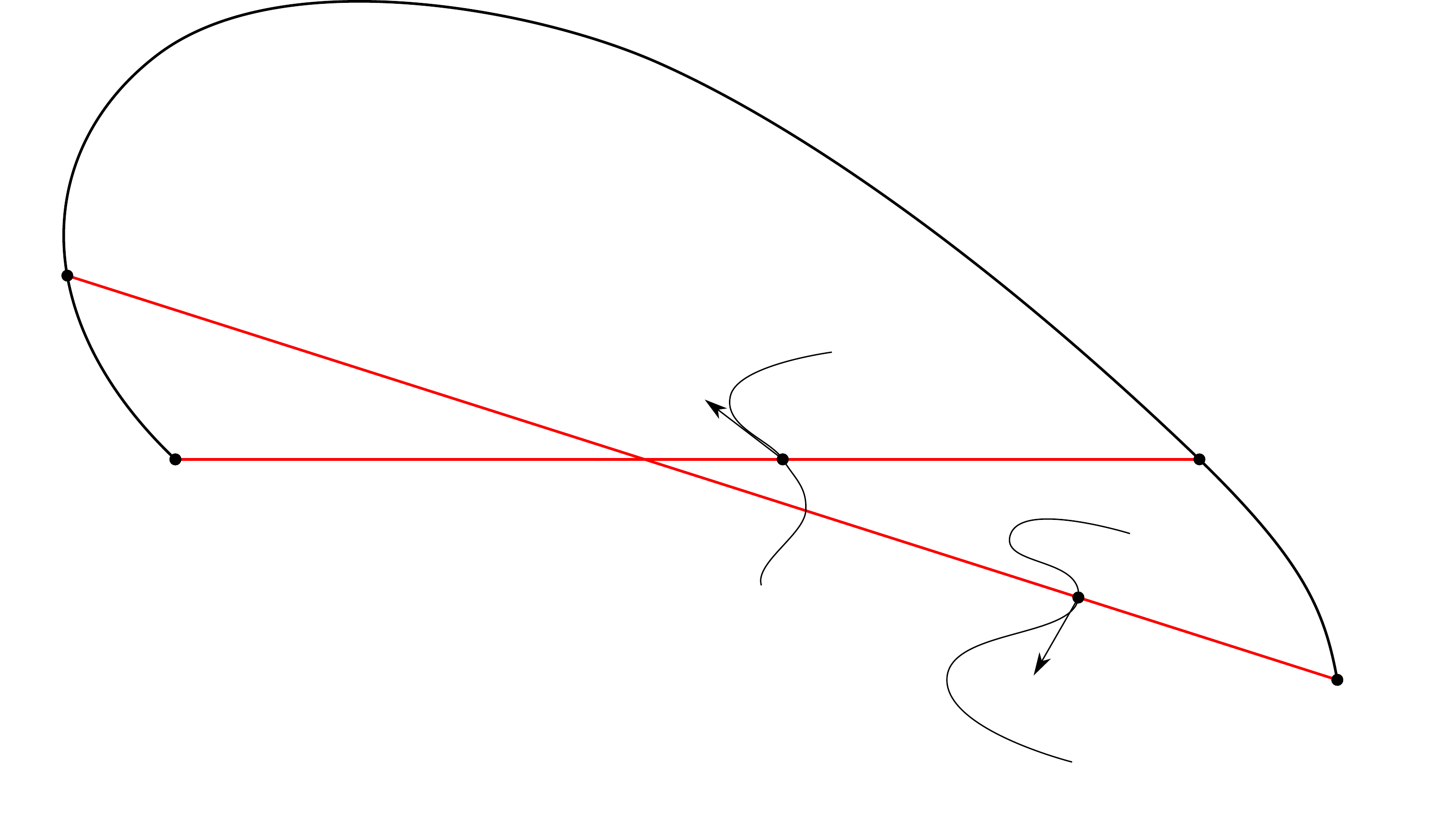\caption{The arc $\wideparen{q_1p_1}$ and the points $p_1^\sharp$, $p_1^\dagger$, $q_1^\sharp$, $q_1^\dagger$.
}\label{fig_sharp}\par
\end{figure}


Clearly $\overline{p_1p_1^\dagger}$ and $\overline{q_1q_1^\dagger}$ do not intersect, and therefore the arcs $\wideparen{p_1p_1^\dagger}$ and $\wideparen{q_1^\dagger q_1}$ are disjoint. In fact, since the set $\Omega\gamma$ is homeomorphic to any disk we can represent $\gamma$ as the unit circle counterclockwise oriented. See for instance Figure \ref{fig_chord}. Regardless the position of the points $q_1$ and $p_1$, if we look on the arc $\wideparen{p_1q_1}$ the point $q_1^\dagger$ can follow $p_1^\dagger$ as in Figure \ref{fig_chord}$(a)$ or precedes as in Figure \ref{fig_chord}$(b)$. We represent $\overline{p_1p_1^\dagger}$ in green and $\overline{q_1q_1^\dagger}$ in blue (they are segment in $\Omega\gamma$ but after the homeomorphism they become  continuous curves). Since we know that they do not intersect (a property which is invariant under homeomorphism)  then Figure \ref{fig_chord}$(a)$ is the only possible configuration. Therefore $\wideparen{p_1p_1^\dagger}$ and $\wideparen{q_1^\dagger q_1}$ are disjoint.

\begin{figure}
\def\svgwidth{5.in}
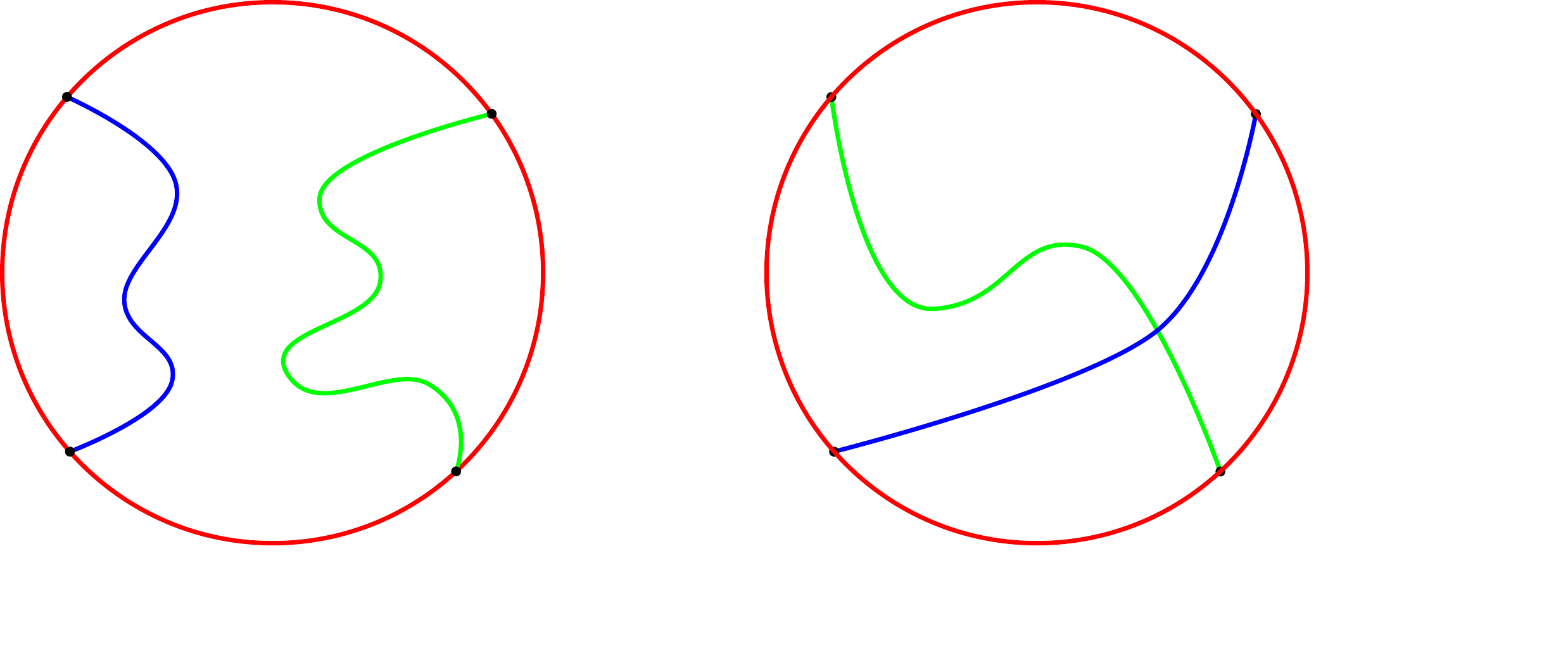\caption{
The set $\Omega_\gamma$ is homeomorphic to a disk. This figure shows the image through such an homeomorphism of the curve $\gamma$ in red, the chords $\overline{p_1p_1^\dagger}$ in green and $\overline{q_1q_1^\dagger}$ in blue.}\label{fig_chord}\par
\end{figure}

We claim that there exists a non void maximal convex arc, say $\Gamma_p$, having not empty intersection with $\wideparen{p_1p_1^\dagger}$ and a non void maximal convex arc, say $\Gamma_q$, having not empty intersection with $\wideparen{q_1^\dagger q_1}$. Moreover we claim that such two arcs are distinct.  
With this respect we observe that if they exists they are certainly distinct. If for instance $\Gamma_p\subset \wideparen{p_1p_1^\dagger}$ or $\Gamma_q\subset \wideparen{q_1^\dagger q_1}$, then $\Gamma_p\not\equiv\Gamma_q$ because $\wideparen{p_1p_1^\dagger}\cap\wideparen{q_1^\dagger q_1}=\emptyset$. In order to have $\Gamma_p\equiv\Gamma_q$ one should have $\wideparen{p_1^\dagger q_1^\dagger}\subset\Gamma_p$. This means that the arc $\wideparen{p_1^\dagger q_1^\dagger}$ is a convex arc which contradict the fact that the rotation angle of $\wideparen{p_1^\dagger q_1^\dagger}$ is negative.

Now, let us restrict our attention on $\wideparen{p_1p_1^\dagger}$. We want to find a non maximal convex arc having not empty intersection with $\wideparen{p_1p_1^\dagger}$. The same argument can be then repeated for $\wideparen{q_1^\dagger q_1}$.

We observe that $\wideparen{p_1p_1^\dagger}$ has non empty intersection with the convex arc which follows $\wideparen{q_1p_1}$ in the counterclockwise orientation.

We denote such a maximal convex arc $\wideparen{q_2p_2}$. There are two possibilities:
\begin{itemize}
\item[Case 1:] $p_2$ does not belong to $\wideparen{p_1p_1^\dagger}$.
\item[Case 2:] $p_2$ belongs to $\wideparen{p_1p_1^\dagger}$.
\end{itemize} 

\begin{figure}
\def\svgwidth{6.in}
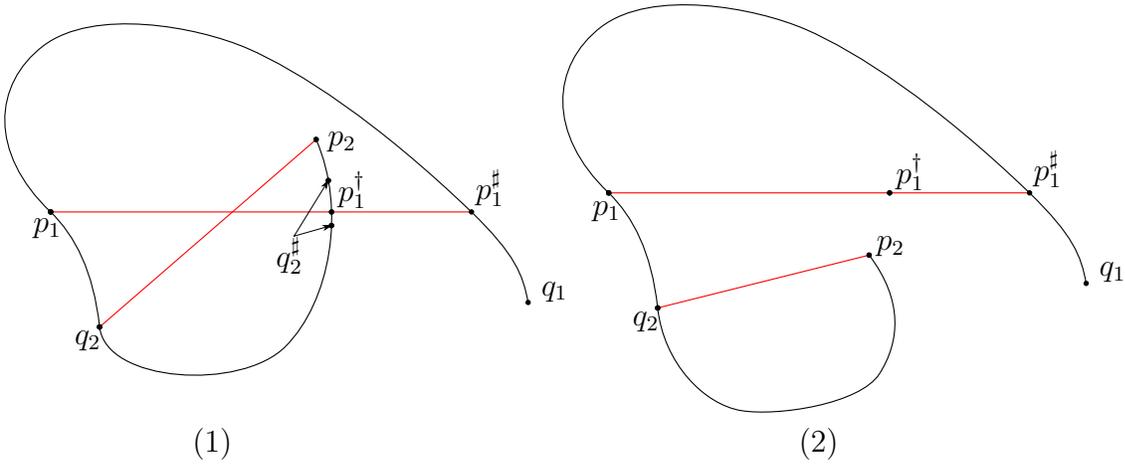\caption{In Case 1 $p_2$ does not belong to $\wideparen{p_1p_1^\dagger}$. The point $q_2^\sharp$ may precede or follow $p_1^\dagger$.
In case (2) $p_2$ belongs to $\wideparen{p_1p_1^\dagger}$}\label{fig_two_cases}
\end{figure}

In Case 1, see Figure \ref{fig_two_cases}(1), if the point $p_1^\dagger$ does not belong to the arc $\wideparen{q_2q_2^\sharp}$, then $\overline{q_2q_2^\sharp}$ intersects $\gamma$ only at the endpoints. If $p_1^\dagger$ belongs to the arc $\wideparen{q_2q_2^\sharp}$ then the arc $\wideparen{q_2q_2^\sharp}$ is nested into the arc $\wideparen{p_1^\sharp p_1}$ and therefore we can use Proposition \ref{prop_nested}. In both cases $\wideparen{q_2q_2^\sharp}$ holds a convex set.

In case (2), see Figure \ref{fig_two_cases}(2), the arc $\wideparen{q_2p_2}$ might be void or not. If it is non void then there exist an arc $\wideparen{q'_2p'_2}\subset \wideparen{q_2p_2}$ which holds a convex set and the claim is proved. If on the contrary the arc $\wideparen{q_2p_2}$ is void then we can iterate the argument. We can indeed find a chord $\overline{p_2p_2^\dagger}$ such that the arc $\wideparen{p_2p_2^\dagger}$ (contained in $\wideparen{p_1p_1^\dagger}$) has non empty intersection with the maximal convex arc which follows $\wideparen{q_2p_2}$ in the counterclockwise orientation and so on.

Since $\# \gamma$ is finite we can iterate the the argument only a finite number of steps after which necessarily for some $i\in \mathbb{N}$ there exists a non void arc $\wideparen{p_{i+1}q_{i+1}}$ with non empty intersection with $\wideparen{p_{i}p_{i}^\dagger}$ and the claim is proved.
   
\subsection*{End of the proof for final shape in $\mathcal{C}_\pi$}
In this section we show that all curves in $\mathcal{C}_\pi$ have always at least one arc holding a convex set. This fact, together with Corollary \ref{c_two_convex} concludes the proof of Theorem \ref{main_th}. 

Let $\gamma\in \mathcal{C}_\pi$, by elementary geometric arguments the curve must have at least one maximal convex arc such that the integral of the curvature exceeds the value $\pi$. If there exists just one such arc then it is necessarily a non void arc. We assume therefore that there exists more than one such arc, and then at least one of them doesn't have the cusp as endpoint. Observe that the cusp might be the end point of two concave arcs, or possibly one concave arc and one convex arc, but definitely not the endpoint of two convex arcs. Therefore we assume to start from a maximal convex arc $\wideparen{q_1p_1}$ and the nontrivial case is the case where such an arc is also void. We proceed as in the previous paragraph. Since $\wideparen{q_1p_1}$ is void then $\gamma$ crosses the segment $\overline{p_1p_1^\sharp}$ somewhere in between the two endpoints. There exists therefore $p_1^\dagger$ on $\overline{p_1p_1^\sharp}$ such that $\overline{p_1p_1^\dagger}$ is a chord for $\gamma$. In the same way we can find $q_1^\dagger$ on $\overline{q_1q_1^\sharp}$ such that $\overline{q_1q_1^\dagger}$ is a chord for $\gamma$. Clearly $\overline{p_1p_1^\dagger}$ and $\overline{q_1q_1^\dagger}$ do not intersect, and in view of what we observed before (see Figure \ref{fig_chord}) the arcs $\wideparen{p_1p_1^\dagger}$ and $\wideparen{q_1^\dagger q_1}$ are disjoint. However we might be not able to repeat all arguments used before since possibly,  either in $\wideparen{p_1p_1^\dagger}$ or in $\wideparen{q_1^\dagger q_1}$ we might encounter the cusp. Nevertheless the argument has to work in one of the two cases and  this completes the proof.

\subsection*{Uniqueness in Theorem \ref{main_th}}
So far our proof characterizes the circle as the unique curve achieving the equality in \eqref{isopineq} in the restricted class of simple closed $C^1$ and piecewise-$C^2$ curves whose curvature changes sign only a finite number of times. In particular we have uniqueness in the class of analytic curves. We claim that the circle is actually unique in the wider class of simple closed $W^{2,2}$ curves. To this aim we first observe that existence of a minimizer of $E(\cdot)^2 A(\cdot)$ in $W^{2,2}$ is established via approximation. Then we observe that any minimizer in $W^{2,2}$ has the curvature which satisfies \eqref{eq_formal}
and we can conclude that a minimizer is also an analytic curve. Our claim immediately follows.

\bigskip

\noindent {\bf Acknowledgement.} We thank A.Henrot for helpful discussions during a visit to Napoli in 2013. After we had completed this manuscript in November 2014 (arXiv:1411.6100) and sent it to colleagues for comments, we were kindly informed by A.Henrot that he and D.Bucur had answered the open problem addressed in the present paper with a different proof. Their final written proof has appeared in December 2014 on CVGMT (http://cvgmt.sns.it/paper/2582/).
C.Nitsch was financially supported by an Alexander-von-Humboldt grant and by Progetto Star ``SInECoSINE". First and third authors are members of the Gruppo Nazionale per l'Analisi Matematica, la Probabilit\`a e le loro Applicazioni (GNAMPA) of the Istituto Nazionale di Alta Matematica (INdAM).

\end{document}